\documentclass[journal]{IEEEtran}

\usepackage{cite}

\ifCLASSINFOpdf
 \usepackage[pdftex]{graphicx}
 \usepackage{subfigure}
\else 
\fi 
\usepackage{amsmath}
\usepackage{amssymb}
\usepackage{amsthm}

\usepackage{color}      
 \usepackage{mathrsfs}

\usepackage{float} 

\newtheorem{lemma}{\textbf{Lemma}}

\newtheorem{theorem}{\textbf{Theorem}}

\newtheorem{assumption}{\textbf{Assumption}}

\newcommand{\prob}[1]{\text{Pr}\big\{#1\big\}}
\newcommand{\expect}[1]{\mathbb{E}\big\{#1\big\}}
\newcommand{\expectm}[1]{\mathbb{E}\{#1\}}

\newcommand{\bv}[1]{{\boldsymbol{#1} }}
\newcommand{\script}[1]{{{\cal{#1} }}}

\newtheorem{definition}{\textbf{Definition}}

\newcommand{\olac}{\texttt{OLAC}}

\newcommand{\rlc}{\texttt{RLC}}
\newcommand{\plc}{\texttt{PLC}}
\newcommand{\ade}{\texttt{ADE}}
\newcommand{\bp}{\texttt{BP}}
\newcommand{\rhc}{\texttt{RHC}}
\allowdisplaybreaks

 \newenvironment{achievements}{\begin{list}{$\bullet$}{\topsep 0.05pt \itemsep -.1pt}}{\vspace*{1pt}\end{list}}

\allowdisplaybreaks
\begin{document}

\title{Learning-aided Stochastic Network Optimization with State Prediction}

\author{Longbo Huang$^*$, Minghua Chen$^+$, Yunxin Liu$^\dagger$\\
$*$longbohuang@tsinghua.edu.cn, IIIS@Tsinghua University \\
$+$minghua@ie.cuhk.edu.hk, IE@CUHK\\
$\dagger$yunxin.liu@microsoft.com, Microsoft Research Asia
\thanks{This paper was presented in part at the $18$th ACM International Symposium on Mobile Ad Hoc Networking and Computing (MobiHoc),  India, July 2017.}
} 

\maketitle

\begin{abstract}
We investigate the problem of stochastic network optimization in the presence of state prediction and non-stationarity. Based on a novel state prediction model featured with a distribution-accuracy curve, 
we develop the \emph{predictive learning-aided control} (\plc) algorithm, which jointly utilizes historic and predicted network state information for decision making. \plc{} is an online algorithm that 
consists of three key components, namely sequential distribution estimation and change detection, dual learning, and online  queue-based control. 
We show that 
for stationary networks, \plc{} achieves a near-optimal 
utility-delay  tradeoff. For non-stationary networks,   \plc{} obtains an 
utility-backlog tradeoff for distributions that last longer than a time proportional to the square of the prediction error, which is smaller than that needed by Backpressue  \cite{neelynowbook} for achieving the same utility performance. 
Moreover,  \plc{} detects distribution change $O(w)$ slots faster with high probability ($w$ is the prediction size)  and achieves a convergence time faster than that under Backpressure. Our results demonstrate that state prediction  helps (i) achieve faster detection and convergence, and (ii) obtain better utility-delay tradeoffs. They also  quantify the benefits of prediction in  four important   performance metrics, i.e., utility (efficiency), delay (quality-of-service), detection (robustness), and convergence (adaptability), and provide new insight for joint prediction, learning and optimization in  stochastic networks. 

\end{abstract}

\maketitle 


\section{Introduction} \label{section:intro}
Enabled by recent developments in sensing, monitoring, and machine learning methods, utilizing prediction for performance improvement in networked systems has received a growing attention in  both  industry and research. 
%
For instance, recent research works \cite{mobility-predict11}, \cite{straggler-mitigation-nsdi14}, and \cite{predict-mobile-15}  investigate the benefits of utilizing prediction in energy saving, job migration in cloud computing, and video streaming in cellular networks. 
On the industry side, various companies have implemented different ways to take advantage of prediction, e.g., Amazon utilizes prediction for better  package delivery  \cite{amazon-anticipatory-14}  and Facebook enables  prefetching for faster webpage loading \cite{facebook-prefetching}. 
However, despite the continuing success in these attempts,  most existing results in network control and analysis do not investigate the impact of prediction. Therefore, we still lack a thorough theoretical understanding about the \emph{value-of-prediction} in stochastic network control. 
Fundamental questions regarding how prediction should be integrated in network algorithms, the ultimate prediction gains, and how prediction error impacts  performance, remain largely unanswered. 


To contribute to developing a theoretical foundation for utilizing prediction in networks, in this paper, we consider a general constrained stochastic network optimization formulation, and aim to  rigorously quantify the benefits of system state prediction and the impact of prediction error. 
Specifically, we are given a discrete-time stochastic network  with a dynamic state that evolves according to some potentially  non-stationary  probability law. Under each system  state, a control action is chosen and implemented.  The action  generates traffic into network queues but also serves  workload from them. The action also results in a system utility (cost) due to service completion (resource expenditure). The traffic, service, and cost are jointly determined by the action and the system state. 
The objective is to maximize the expected utility (or equivalently, to minimize the cost) subject to traffic/service constraints, given imperfect system state prediction information. 
%


This is a general framework that  models various practical scenarios, for instance, mobile networks, computer networks, supply chains, and smart grids. 
%
However, understanding the impact of prediction in this framework is challenging.  
First,  statistical information of network dynamics is often unknown a-priori. Hence, in order to achieve good performance,  algorithms must be able to quickly learn certain sufficient statistics of the dynamics, and make efficient use of prediction while carefully handling prediction error. 
Second,  system states appear randomly in every time slot. Thus, algorithms  must perform well under such incremental realizations of the randomness. 
Third, quantifying system service quality often involves handling queueing in the system. As a result, explicit connections between control actions  and queues must be established.  
%



There has been a recent effort in developing algorithms that can achieve good utility and delay performance for this general problem without prediction in various settings, for instance, wireless networks, \cite{control-rechargeable-twc10},  \cite{eryilmaz_qbsc_ton07}, \cite{li-sigmetrics15}, \cite{ji-mobihoc16}, processing networks, \cite{zhao-forkandjoin-spn}, \cite{jiang-spn},  cognitive radio, \cite{rahulneelycognitive}, and  the smart grid, \cite{gamal-storage-11}, \cite{rahulneely-storage}. 
%
However,  existing results mostly focus on networks with stationary distributions. They either assume full system statistical information beforehand, or rely on stochastic approximation techniques to avoid the need of such information. 
Works  \cite{huang-learning-sig-14} and \cite{huang-rlc-15} propose schemes to incorporate historic system information into control, but they do not consider prediction.  
Recent results in \cite{tadrous-proactive}, \cite{xu-mm1}, \cite{zhang-predict-mm1-14},  \cite{xu-necessity-future-info-15} and \cite{huang-predictive-14} consider problems with  traffic demand prediction, and \cite{proactive-tadrous-allerton15} jointly considers demand and channel prediction. However,  they  focus either on $M/M/1$-type models, or do not consider queueing, or do not consider the impact of  prediction error. 
Along a different line of work, \cite{ev-charging-xiaojun-infocom15}, \cite{ooc-sig-15}, \cite{ooc-sig-16} and \cite{rchase-eenergy16} investigate the benefit of prediction from the online algorithm design perspective. 
%
Although the results provide new understanding about the effect of prediction, they  do not apply to the general constrained network optimization problem in consideration, where action outcomes are general functions of time-varying network states, queues evolve in a controlled manner, i.e., arrival and departure rates depend on the control policy, and prediction can contain error. 



In this paper, we develop a novel control algorithm for the general framework called \emph{predictive learning-aided control} (\plc). \plc{} is an online scheme that consists of three components, sequential distribution estimation and change detection, dual learning, and online control (see Fig. \ref{fig:alg-demo}). 
\begin{figure}[ht] 
\begin{center}
\vspace{-.15in} 
\includegraphics[width=2.6in, height=1.2in]{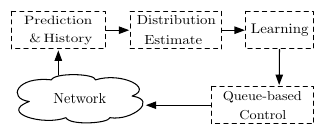}
\vspace{-.15in}
\caption{The \plc{} algorithm contains (i) a distribution estimator that utilizes both historic and predicted information to simultaneously form a distribution estimation and detect distribution change, (ii) a learning component that computes an empirical Lagrange multiplier based on  the estimation, and (iii) a queue-based controller whose decision-making information is augmented by the multiplier. 
}
\label{fig:alg-demo}
\end{center}
\vspace{-.1in}
\end{figure}

The distribution estimator conducts sequential statistical comparisons based on prediction and historic network state records. Doing so efficiently detects changes of the underlying probability distribution and guides us in selecting the right state samples to form distribution estimates. 
The estimated distribution is then fed into a dual learning component to compute an empirical multiplier of an underlying optimization formulation. This multiplier is  further incorporated into the Backpressure (\bp) controller \cite{neelynowbook} to perform realtime network operation.  
Compared to the commonly adopted receding-horizon-control approach (\rhc), e.g., \cite{minghong-igcc-12}, \plc{} provides another way to utilize future state information, which focuses on using the predicted distribution for guiding action selection in the present slot and can be viewed as performing steady-state control under the predicted future distribution. 

We summarize our main contributions as follows.   

i. We propose a general state prediction model featured with a distribution-accuracy curve. Our model captures key factors of several existing prediction models, including window-based  \cite{huang-predictive-14}, distribution-based \cite{tadrous-proactive-shaping}, and filter-based \cite{ooc-sig-16} models. 

ii. We  propose a general constrained network control algorithm called \emph{predictive learning-aided control} (\plc). \plc{} is an online algorithm that 
is applicable to both stationary and non-stationary systems. It jointly performs distribution estimation and change detection, dual learning, and queue-based online control. 

iii. We show that for stationary networks, \plc{} achieves an $[O(\epsilon), O(\log^2(1/\epsilon))]$ utility-delay tradeoff. For non-stationary networks,  \plc{} obtains an $[O(\epsilon), O(\log^2(1/\epsilon)$ $+ \min(\epsilon^{c/2-1}, e_w/\epsilon))]$ utility-backlog tradeoff for distributions that last $\Theta(\frac{\max(\epsilon^{-c}, e_w^{-2})}{\epsilon^{1+a}})$ time, where $e_w$ is the prediction accuracy, $c\in(0, 1)$ and $a>0$ is an $\Theta(1)$ constant (the Backpressue algorithm \cite{neelynowbook} requires an $O(\epsilon^{-2})$ length for the same utility performance with a larger backlog).\footnote{Note that when there is no prediction, i.e., $w=0$ and $e_w=\infty$, we recover previous results of \olac{} \cite{huang-learning-sig-14}. } 

%


iv. We show that for both stationary and non-stationary system dynamics, \plc{} detects distribution change $O(w)$ slots ($w$ is prediction window size) faster with high probability and achieves a fast $O(\min(\epsilon^{-1+c/2}, e_w/\epsilon)+\log^2(1/\epsilon))$ convergence time, which is faster than the $O(\epsilon^{-1+c/2} + \epsilon^{-c})$ time of the \olac{} scheme \cite{huang-learning-sig-14}, and the $O(1/\epsilon)$ time of Backpressure. 

The rest of the paper is organized as follows. In Section \ref{section:examples}, we discuss a few motivating examples.  We set up the  notations in Section \ref{section:notation}, and present the problem formulation in Section \ref{section:model}.  Background information is provided in Section \ref{section:review}. Then, we present \plc{} and its analysis in Sections  \ref{section:plc} and \ref{section:performance}. Simulation results are presented in Section \ref{section:sim}, followed by conclusions in Section \ref{section:conclusion}. 
To facilitate reading, all the proofs are placed in the appendices.

\section{Motivating Examples}\label{section:examples}
In this section, we present a few interesting practical scenarios that fall into our general framework.

\textbf{Matching in sharing platforms}: Consider a Uber-like company that provides ride service to customers. At every time, customer requests enter the system and available cars join to provide service. Depending on the environment condition (state), e.g., traffic condition or customer status, matching customers to drivers can result in different user satisfaction, and affect the revenue of the company (utility). The company gets access to future customer demand and car availability, and system condition information (prediction), e.g., through reservation or machine learning tools. The objective is to optimally match customers to cars so that the  utility is maximized, e.g., \cite{taxi-recommend-kdd14} and \cite{matching-queue-ton}. 

\textbf{Energy optimization in mobile networks}:  
Consider a  base-station (BS) sending traffic to a set of mobile users. 
The channel conditions (state) between users and the BS are time-varying. Thus, the BS needs different amounts of power for packet transmission (cost) at different times. 
Due to higher layer application requirements, the BS is required to deliver packets to users at  pre-specified rates. On the other hand, the BS can predict future user locations  in some short period of time, from which it can estimate future channel conditions (prediction). 
The objective of the BS is to jointly optimize power allocation and scheduling among users, so as to minimize  energy consumption, while meeting the rate requirements, e.g., \cite{eryilmaz_qbsc_ton07}, \cite{rahulneelycognitive}. Other factors such as energy harvesting, e.g.,  \cite{tap-energy-ton14}, can also be incorporated in the formulation.

\textbf{Resource allocation in cloud computing}: Consider an operator, e.g., a dispatcher,  assigning computing jobs to servers for processing. 
The job arrival process is time-varying (state), and available processing capacities at servers are also dynamic (state),  e.g., due to background processing. 
Completing users' job requests brings the operator reward (utility). The operator may also have information regarding future job arrivals and service capacities (prediction). 
The goal is to allocate resources and to balance the loads properly, so as to maximize system utility. 
This example can be extended to capture other factors such as rate scaling \cite{yuan-power-tpds-14} and data locality constraints \cite{wang-mapreduce-ton}.


%
%


In these examples and  related works, not only can the  state statistics be potentially non-stationary, but the systems also often get access to certain (possibly imperfect) future state information through various prediction techniques. 
These features make the problems different from existing settings considered, e.g., \cite{eryilmaz_qbsc_ton07} and \cite{rahulneely-storage}, and require different approaches for both algorithm design and analysis. 

%
%

\section{Notations}\label{section:notation}
$\mathbb{R}^n$ denotes the $n$-dimensional Euclidean space. $\mathbb{R}^n_+$ ($\mathbb{R}^n_-$) denotes the non-negative (non-positive) orthant. Bold symbols $\bv{x}=(x_1, ..., x_n)$ denote vectors in $\mathbb{R}^n$.  $w.p.1$ denotes ``with probability $1$.''  $\|\cdot\|$ denotes the Euclidean norm. For a sequence  $\{y(t)\}_{t=0}^{\infty}$,  $\overline{y}=\lim_{t\rightarrow\infty}\frac{1}{t}\sum_{\tau=0}^{t-1}\expect{y(\tau)}$ denotes its average (when exists). $\bv{x}\succeq\bv{y}$ means $x_j\geq y_j$ for all $j$. 
For  distributions $\bv{\pi}_1$ and $\bv{\pi}_2$,  $\|\bv{\pi}_1 - \bv{\pi}_2\|_{TV}=\sum_i |\pi_{1i} - \pi_{2i}|$ denotes the total variation distance. 

\section{System Model}\label{section:model}
Consider a  controller that operates a network with the goal of minimizing the time average cost, subject to the queue stability constraint. The network  operates in slotted time, i.e., $t\in\{0,1,2,...\}$, and  there are $r\geq1$ queues in the network. 

\subsection{Network state}\label{section:state}
In every slot $t$,  $S(t)$ denotes the current network state, which summarizes current network parameters, such as a vector of conditions for each network link, or a collection of other relevant information about the current network channels and arrivals. 
 $S(t)$ is independently distributed across time, and each realization is drawn from a state space of $M$ distinct states denoted as $\mathcal{S} = \{s_1, s_2, \ldots, s_M\}$.\footnote{The independent assumption is made to facilitate presentation and understanding. The results in this paper can likely be generalized to systems where $S(t)$ evolves according to general time inhomogeneous Markovian dynamics. } 
We denote $\pi_{i}(t)=\prob{S(t)=s_i}$ the probability of being in state $s_i$ at time $t$ and denote $\bv{\pi}(t)=(\pi_{1}(t), ..., \pi_{M}(t))$  the state distribution. 
The network controller can observe $S(t)$ at the beginning of every slot $t$, but the $\pi_{i}(t)$ probabilities are unknown. 
We assume that each $\bv{\pi}(t)$ stays unchanged for multiple timeslots, and 
denote $\{t_k, k=0, 1, ...\}$ the starting point of the $k$-th constant distribution interval $\mathcal{I}_k$, i.e., $\bv{\pi}(t) = \bv{\pi}_k$ for all $t\in\mathcal{I}_k\triangleq\{t_k, t_{k+1}-1\}$. The length of $\mathcal{I}_k$ is denoted by $d_k\triangleq t_{k+1}-t_k$. 

\subsection{State prediction} \label{section:prediction}
At every time slot, the  operator gets access to a prediction module, e.g., a machine learning algorithm,  which provides prediction of future network states. 
%
Different from recent works, e.g.,  \cite{ooc-sig-15},  \cite{ooc-sig-16} and  \cite{lingwen-realtime-eenergy13}, which  assume prediction models on individual states,   we assume that the prediction module outputs a sequence of predicted distributions $\mathcal{D}_w(t)\triangleq \{\hat{\bv{\pi}}(t), \hat{\bv{\pi}}(t+1), ..., \hat{\bv{\pi}}(t+w)\}$, where $w+1$ is the prediction window size. 
Moreover, the prediction quality is characterized by a distribution-accuracy curve $\{e(0), ..., e(w)\}$ as follows.  For every $0\leq k\leq w$,  $\hat{\bv{\pi}}(t+k)$  satisfies: 
\begin{eqnarray}
||\hat{\bv{\pi}}(t+k) -  \bv{\pi}(t+k)||_{TV}\leq e(k), \,\,\forall\,\, k. \label{eq:predict-model} 
\end{eqnarray} 
That is, the predicted distribution at time $k$ has a total-variation error bounded by some $e(k)\geq0$.\footnote{We focus on state distribution prediction instead of predicting individual states. In this case, it makes sense to assume a deterministic upper bound of the difference because we are dealing with distributions.} Note that $e(k)=0$ for all $0\leq k\leq w$ corresponds to a perfect predictor, in that it predicts the exact distribution in  every slot. 
We assume  the $\{e(0), ..., e(w)\}$ curve is \emph{known} to the operator and denote 
$e_{w}\triangleq\frac{1}{w+1}\sum_{k=0}^we(k)$ the average prediction error.  

It is often possible to achieve prediction guarantees as in (\ref{eq:predict-model}), e.g., by adopting a maximum likelihood estimator (Section 5.1 in  \cite{walrand-peecs}) based on historic state data. 
Also note that the prediction model in (\ref{eq:predict-model}) is general and captures key features of several existing prediction models: 
(i) the exact distribution  prediction model in \cite{tadrous-proactive-shaping}, where the future demand distribution is known ($e(k)=0$ for all $k$), 
(ii) the window-based prediction model, e.g.,  \cite{huang-predictive-14}, where each $\hat{\bv{\pi}}(t+k)$ corresponds to the indicator for the true state, and 
%
(iii) the error-convolution prediction model  in \cite{lingwen-realtime-eenergy13},  \cite{ooc-sig-15} and \cite{ooc-sig-16}, which captures key features of  the Wiener filter and Kalman filter. 
\subsection{The cost, traffic, and service}\label{subsection:costtrafficservice}
At each time $t$, after observing $S(t)=s_i$, the controller chooses an action $x(t)\in\mathcal{X}_i$. The set $\mathcal{X}_i$ is called the feasible action set for network state $s_i$ and is assumed to be time-invariant and compact for all $s_i\in\mathcal{S}$. 
The cost, traffic, and service generated by the action $x(t)=x_i$ are as follows:
\begin{achievements}
\item[(a)] The chosen action has an associated cost given by the cost function $f(t)=f(S(t), x(t))=f(s_i, x_i): \mathcal{X}_i\mapsto \mathbb{R}_+$ (or $\mathcal{X}_i\mapsto\mathbb{R}_-$ in reward maximization problems).\footnote{We use cost and utility interchangeably in this paper.}
\item[(b)] The amount of traffic generated by the action to queue $j$ is determined by the traffic function $A_j(t)=A_j(S(t), x(t))=A_{j}(s_i, x_i): \mathcal{X}_i\mapsto \mathbb{R}_{+}$, in units of packets.
\item[(c)] The amount of service allocated to queue $j$ is given by the rate function 
$\mu_j(t)=\mu_j(S(t), x(t))=\mu_{j}(s_i, x_i): \mathcal{X}_i\mapsto \mathbb{R}_{+}$, in units of packets.
 \end{achievements}
Here $A_j(t)$ can include both exogenous arrivals from outside the network to queue $j$, and endogenous arrivals from other queues, i.e.,  transmitted packets from other queues to queue $j$. 
We assume the functions $-f(s_i, \cdot)$, $\mu_{j}(s_i, \cdot)$ and $A_{j}(s_i, \cdot)$ are time-invariant, their magnitudes are uniformly upper bounded by some constant $\delta_{\max}\in(0,\infty)$ for all $s_i$, $j$, and they are known to the  operator.
Note that this formulation is general and models many network problems, e.g., \cite{eryilmaz_qbsc_ton07}, \cite{rahulneely-storage}, and \cite{ying_wmshortest_infocom09}. 

\subsection{Problem formulation}
\label{section:queuenotation}
Let $\bv{q}(t)=(q_1(t), ..., q_r(t))^T\in\mathbb{R}^r_{+}$, $t=0, 1, 2, ...$ be the queue backlog vector  process of the network, in units of packets. We assume the following queueing dynamics: %
\begin{eqnarray}
q_j(t+1)=\max\big[q_j(t)-\mu_j(t)+A_j(t), 0\big], \quad\forall j,\label{eq:queuedynamic}
\end{eqnarray}
and $\bv{q}(0)=\bv{0}$. By using (\ref{eq:queuedynamic}), we assume that when a queue does not have enough packets to send, null packets are transmitted, so that the number of packets entering $q_j(t)$ is equal to $A_j(t)$. We adopt the following notion of queue stability \cite{neelynowbook}:
\begin{eqnarray}
\overline{q}_{\text{av}}\triangleq
\limsup_{t\rightarrow\infty}\frac{1}{t}\sum_{\tau=0}^{t-1}\sum_{j=1}^{r}\expect{q_j(\tau)}<\infty.\label{eq:queuestable}
\end{eqnarray}
We use $\Pi$ to denote an action-choosing policy, and use $f^{\Pi}_{\text{av}}$ to denote its  time average cost, i.e., 
\begin{eqnarray}
f^{\Pi}_{\text{av}}\triangleq
\limsup_{t\rightarrow\infty}\frac{1}{t}\sum_{\tau=0}^{t-1}\expect{f^{\Pi}(\tau)},\label{eq:timeavcost}
\end{eqnarray}
where $f^{\Pi}(\tau)$ is the cost incurred at time $\tau$ under policy $\Pi$. We call an action-choosing  policy \emph{feasible} if at every time slot $t$ it only chooses actions from the feasible action set $\mathcal{X}_i$ when $S(t)=s_i$.  We then call a feasible action-choosing  policy under which (\ref{eq:queuestable}) holds a \emph{stable} policy. 

In every slot, the network controller observes the current network state and prediction, and chooses a control action, with the goal of minimizing the time average cost subject to network stability. 
This goal can be mathematically stated as:\footnote{When $\bv{\pi}(t)$ is time-varying, the optimal system utility needs to be defined carefully. We will specify it when discussing the corresponding results.}
\begin{eqnarray*}
\textbf{(P1)}\,\,\, \bv{\min_{\Pi} \,  f^{\Pi}_{\text{av}}, \,\, \text{s.t.}\,  (\ref{eq:queuestable})}.
\end{eqnarray*}
In the following, we call \textbf{(P1)} \emph{the stochastic problem}, and we use $f^{\bv{\pi}}_{\text{av}}$ to denote its optimal solution given a fixed distribution $\bv{\pi}$. 
It can be seen that the examples in Section \ref{section:examples} can all be modeled by our stochastic problem framework. 

Throughout our paper, we make the following assumption.  
\begin{assumption}\label{assumption:bdd-LM}
For every system distribution $\bv{\pi}_k$, there exists a constant $\epsilon_k=\Theta(1)>0$ such that for any valid state distribution $\bv{\pi}' = (\pi'_{1}, ..., \pi'_{M})$ with $\|\bv{\pi}' - \bv{\pi}_k \|_{TV}\leq \epsilon_k$, there exist a set of actions $\{x^{(s_i)}_z\}_{i=1,..., M}^{z=1,2, ..., \infty}$ with $x^{(s_i)}_z\in\mathcal{X}_i$ and variables $\vartheta^{(s_i)}_z\geq0$ for all $s_i$ and $z$ with $\sum_z\vartheta^{(s_i)}_z=1$ for all $s_i$ (possibly depending on $\bv{\pi}'$), such that:
\begin{eqnarray}
&&\sum_{s_i}\pi'_{i}\big\{\sum_z\vartheta^{(s_i)}_z[A_{j}(s_i, x^{(s_i)}_z)-\mu_{j}(s_i, x^{(s_i)}_z)]\big\}\nonumber\\
&&\qquad\qquad\qquad\leq -\eta_0,\,\,\forall\,j, \label{eq:slackness}
\end{eqnarray}
where $\eta_0=\Theta(1)>0$ is independent of $\bv{\pi}'$. $\Diamond$
\end{assumption} 
Assumption \ref{assumption:bdd-LM} corresponds to  the ``slack'' condition commonly assumed   in the literature with $\epsilon_k=0$, e.g., \cite{ying_wmshortest_infocom09} and  \cite{buisrikant_infocom09}.\footnote{Note that $\eta_0\geq0$ is a necessary condition  for network stability \cite{neelynowbook}.}  
 With $\epsilon_k>0$, we assume that when two systems are relatively close to each other (in terms of $\bv{\pi}$), they can both be stabilized by some (possibly different) randomized control policy that results in the same slack. 

\subsection{Discussion of the model}\label{subsection:model-discussion}
Two key differences between our model and previous ones include (i)  $\bv{\pi}(t)$ itself can be time-varying and (ii) the operator gets access to a prediction window $\mathcal{W}_w(t)$ that contains imperfect prediction. 
These two extensions are important to the current network control literature. 
First,  practical systems are often non-stationary. Thus, system dynamics can 
have time-varying distributions. Thus, it is  important to have efficient algorithms to automatically adapt to the changing environment. Second, prediction has recently been made increasingly accurate in various contexts, e.g., user mobility in cellular network and harvestable energy availability in wireless systems, by data collection and machine learning tools. Thus, it is  critical to understand the fundamental benefits and limits of prediction, and its optimal usage. 
Finally, note that the convexity of the problem (\textbf{P1}) depends largely on the structure of the feasible action sets. In the case, when all feasible action sets are convex, it can be shown that the resulting problem is convex using a similar argument as that in \cite{neelyenergy}.

\vspace{-.02in}
\section{The Deterministic  Problem}\label{section:review} 
For our later algorithm design and analysis,  we define the \emph{deterministic problem} and its dual problem \cite{huangneely_dr_tac}. 
%
Specifically, the  deterministic problem for a given distribution $\bv{\pi}$ is  defined as follows \cite{huangneely_dr_tac}:
\begin{eqnarray}
\hspace{-.2in}&& \min:   V\sum_{s_i}\pi_{i}f(s_i, x^{(s_i)})\label{eq:primal}\\
\hspace{-.2in}&&\quad\text{s.t.}\,\,\, \sum_{s_i}\pi_{i} [ A_j(s_i, x^{(s_i)})- \mu_j(s_i, x^{(s_i)})] \leq 0,\,\,\forall\, j,\nonumber\\
\hspace{-.2in}&& \qquad\quad  x^{(s_i)}\in \mathcal{X}_i\quad \forall\, i=1, 2, ..., M. \nonumber
\end{eqnarray}
Here  the minimization is taken over $\bv{x}\in\prod_i\mathcal{X}_i$, 
where $\bv{x}=(x^{(s_1)}, ..., x^{(s_M)})^T$, and $V\geq1$ is a positive constant introduced for later analysis. The dual problem of (\ref{eq:primal}) can be obtained as follows:
\begin{eqnarray}
\max:\,\,\, g(\bv{\gamma}, \bv{\pi}),\quad \text{s.t.}\,\,\, \bv{\gamma}\succeq\bv{0},\label{eq:dualproblem}
\end{eqnarray}
where $g(\bv{\gamma}, \bv{\pi})$ is the dual function for problem (\ref{eq:primal}) 
and is defined as:
\vspace{-.06in}
\begin{eqnarray}
\hspace{-.3in}&&g(\bv{\gamma}, \bv{\pi})=\inf_{x^{(s_i)}\in \mathcal{X}_i}\sum_{s_i}\pi_{i}\bigg\{Vf(s_i, x^{(s_i)})\label{eq:dual_separable}\\
\hspace{-.3in}&&\qquad\qquad\qquad\qquad+\sum_j\gamma_j\big[A_j(s_i, x^{(s_i)})- \mu_j(s_i, x^{(s_i)})\big]\bigg\}.\nonumber
\end{eqnarray}
$\bv{\gamma}=(\gamma_1, ..., \gamma_r)^T$ is the  Lagrange multiplier of (\ref{eq:primal}). It is well known that $g(\bv{\gamma}, \bv{\pi})$ in (\ref{eq:dual_separable}) is concave in the vector $\bv{\gamma}$ for all $\bv{\gamma}\in\mathbb{R}^r$. Hence, the problem (\ref{eq:dualproblem}) can usually be solved efficiently, e.g., using dual subgradient methods \cite{bertsekasoptbook}. 
If the cost functions and rate functions are separable over different network components, the problem also admits distributed solutions  \cite{bertsekasoptbook}. 
We use $\bv{\gamma}^*_{\bv{\pi}}$ to denote the optimal multiplier corresponding to a given $\bv{\pi}$ and sometimes omit the subscript when it is clear. 
Denote $g_{\bv{\pi}}^*$  the  optimal value of (\ref{eq:dualproblem}) under a fixed distribution $\bv{\pi}$. It was shown in \cite{huangneely_qlamarkovian} that: 
\begin{eqnarray}
f^{\bv{\pi}}_{\text{av}} = g_{\bv{\pi}}^*. \label{eq:primal-dual-relation}
\end{eqnarray} 
That is, $g_{\bv{\pi}}^*$ characterizes the optimal time average cost of the stochastic problem. 
For our analysis, we make the following assumption on the $g(\bv{\gamma}, \bv{\pi}_k)$ function. 
\begin{assumption}\label{assumption:unique} For every system distribution $\bv{\pi}_k$, $g(\bv{\gamma}, \bv{\pi}_k)$ has a unique optimal solution $\bv{\gamma}_{\bv{\pi}_k}^*\neq\bv{0}$ in $\mathbb{R}^r$. $\Diamond$
\end{assumption} 
Assumption \ref{assumption:unique} holds for many network utility optimization problems, e.g., \cite{eryilmaz_qbsc_ton07},  \cite{huangneely_dr_tac} and \cite{jia-infocom-16}.

 \section{Predictive Learning-aided Control}\label{section:plc}
In this section, we present the \emph{predictive learning-aided control} algorithm  (\plc). 
\plc{} contains three main components: a distribution estimator, a learning component, and an online queue-based controller. Below, we first present the estimation part. Then, we present the \plc{} algorithm. 

\subsection{Distribution estimation and change detection}
Here we specify the distribution estimator. 
%
The idea is to first combine the prediction in $\mathcal{W}_w(t)$ with historic state information to form an  \emph{average} distribution, and then perform statistical comparisons for change detection. 
 We call this module the \emph{average distribution estimate} (\ade). 
 
Specifically, \ade{} maintains two windows $\mathcal{W}_m(t)$ and $\mathcal{W}_d(t)$  to store network state samples, where $\mathcal{W}_d(t)$ roughly contains the most recent $d$ state samples, and $\mathcal{W}_m(t)$ contains at most $T_l$ state samples after $\mathcal{W}_d(t)$. The formal definition of them are given below, i.e.,
\begin{eqnarray}
\hspace{-.2in} \mathcal{W}_d(t) &=& \{b^s_d(t), ...,  b^e_d(t)\}, \label{eq:wd-def}\\
\hspace{-.2in}  \mathcal{W}_m(t)&=&\{ b_{m}(t), ..., \min[b_d^s(t), b_{m}(t) +T_l ] \}. \label{eq:wm-def}
\end{eqnarray}
Here $b^s_d(t)$ and $b_m(t)$ mark the beginning slots of $\mathcal{W}_d(t)$ and $\mathcal{W}_m(t)$, respectively, and $b^e_d(t)$ marks the end of $\mathcal{W}_d(t)$. Ideally, $\mathcal{W}_d(t)$ contains the most recent $d$ samples (including the prediction) and $\mathcal{W}_m(t)$ contains $T_l$  subsequent samples (where $T_l$ is a pre-specified number). 
%
We denote  $W_m(t)=|\mathcal{W}_m(t)|$ and $W_d(t)=|\mathcal{W}_d(t)|$.  
Without loss of generality, we assume that $d\geq w+1$. 
This assumption is made because, $d$ grows with our control parameter $V$ while prediction power is often limited in practice.  
We also denote $\mathcal{W}_w(t) \triangleq \{t, ..., t+w\}$.

We use $\hat{\bv{\pi}}^d(t)$ and $\hat{\bv{\pi}}^m(t)$ to denote the empirical distributions of $\mathcal{W}_d(t)$ and $\mathcal{W}_m(t)$, i.e.,\footnote{Note that this is only one way to utilize the samples. Other methods such as EWMA can also be applied when appropriate.}  
\begin{eqnarray*}
\hspace{-.1in}\hat{\pi}_i^d(t) &=&\frac{1}{d} \bigg(\sum_{\tau=(t+w-d)_+}^{t-1}1_{[S(\tau) = s_i]} + \sum_{\tau\in \mathcal{W}_w(t)} \hat{\pi}_i(\tau) \bigg)\\ 
\hspace{-.1in}\hat{\pi}_i^m(t) &=& \frac{1}{W_m(t)}\sum_{\tau\in\mathcal{W}_m(t)}1_{[S(\tau) = s_i]}. 
\end{eqnarray*}
That is, $\hat{\bv{\pi}}^d(t)$ is the average of the empirical distribution of the ``observed'' samples in $\mathcal{W}_d(t)$ and the predicted distribution, whereas $\hat{\bv{\pi}}^m(t)$ is the empirical distribution of $\mathcal{W}_m(t)$. 

The formal procedure of \ade{} is as follows (parameters $T_l, d, \epsilon_d$ will be  specified later). 
%
%

\underline{Average Distribution Estimate (\ade$(T_l, d, \epsilon_d)$):} Initialize $b^s_d(0)=0$, $b^e_d(0)=t+w$ and $b_m(0)=0$, i.e.,  $\mathcal{W}_d(t)=\{0, ..., t+w\}$ and $\mathcal{W}_m(t)=\phi$. At every time $t$, update $b^s_d(t)$, $b^e_d(t)$ and $b_m(t)$ as follows: 
\begin{enumerate}
\item[(i)]  If  $W_m(t)\geq d$ and $||\hat{\bv{\pi}}^d(t) - \hat{\bv{\pi}}^m(t)||_{TV} > \epsilon_d$, set $b_m(t)=t+w+1$ and $b^s_d(t)=b^e_d(t)=t+w+1$. 
\item[(ii)] If $W_m(t) = T_l$ and there exists $k$ such that $||\hat{\bv{\pi}}(t+k) - \hat{\bv{\pi}}^m(t)||_{TV} >  e(k) +\frac{2M\log(T_l)}{\sqrt{T_l}}$, set $b_m(t)=b^s_d(t)=b^e_d(t)=t+w+1$. Mark $t+w+1$ a  reset point. 

\item[(iii)] Else if $t\leq b_d^s(t-1)$, $b_m(t)=b_m(t-1)$, $b^s_d(t) =b^s_d(t-1)$, and $b^e_d(t) =b^e_d(t-1)$.\footnote{This step is evoked after we set 
 $b_m(t')=b_d^s(t')=t'+w+1\geq t$ for some time $t'$, in which case we the two windows remain unchanged until $t$ is larger than $t'+w+1$.} 
\item[(iv)] Else set $b_m(t)=b_m(t-1)$, $b^s_d(t) = (t+w-d)_+$ and $b^e_d(t) = t+w$. 
\end{enumerate} 
%
Output an estimate at time $t$ as follow: 
\begin{eqnarray}
\bv{\pi}_{a}(t) =\left\{\begin{array}{cc} 
\hat{\bv{\pi}}^m(t) &  \text{if} \quad W_m(t)\geq T_l\\
 \frac{1}{w+1}\sum_{k=0}^{w}\hat{\bv{\pi}}(t+k) &  \text{else} 
\end{array}\right. \Diamond
\end{eqnarray} 
The idea of  \ade{} is shown in Fig. \ref{fig:ade}. 
 \begin{figure}[ht]
\begin{center}
\vspace{-.1in}
\includegraphics[width=3.4in, height=0.9in]{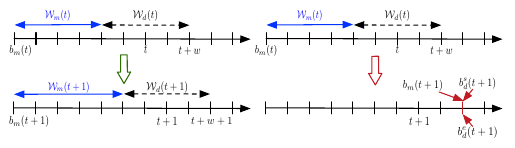}
\vspace{-.2in}
\caption{Evolution of  $\mathcal{W}_m(t)$ and $\mathcal{W}_d(t)$. (Left) No change detected: The samples satisfy the criteria in Steps (i) and (ii). Thus,   $\mathcal{W}_d(t)$ advances by one slot and $\mathcal{W}_m(t)$ increases its size by one. 
(Right) Change detected: The samples violate at least one of the conditions in (i) and (ii). Thus, both windows set their start and end points to $t+w+1$.}
\label{fig:ade}
\end{center}
\vspace{-.1in}
\end{figure}

The intuition of \ade{} is that if the environment is changing over time, we should rely on prediction for control. 
Else if the environment is stationary, then one should use the average distribution learned over time to combat the potential prediction error   that may affect  performance. 
 $T_l$ is introduced to ensure the accuracy of the empirical distribution and can be regarded as the confidence-level given to the distribution stationarity. 
A couple of technical remarks are also ready. (a) The term $ 2M\log(T_l) /\sqrt{T_l}$ is to compensate the inevitable deviation of $\hat{\bv{\pi}}^m(t)$ from the true value due to randomness. 
(b) In $\mathcal{W}_m(t)$, we only use the first $T_l$ historic samples. Doing so avoids random oscillation in  estimation and facilitates  analysis.

Prediction is used in two ways in \ade. First, it is used in step (i) to decide whether the empirical distributions match (average prediction). Second, it is used to check whether prediction is consistent with the history (individual prediction). 
The reason for having this two-way utilization is to accommodate general prediction types. 
For example, suppose each $\hat{\bv{\pi}}(t+k)$ denotes the indicator for state $S(t+k)$, e.g., as in the look-ahead window model \cite{huang-predictive-14}. Then, step (ii) is loose since $e(k)$ is large, but step (i) will be useful. If $\hat{\bv{\pi}}(t+k)$ gets closer to the true distribution, both steps will be useful. 

%
%


\subsection{Predictive learning-aided control} 
We are now ready to present the \plc{} algorithm (shown in Fig. \ref{fig:alg-demo}). The formal description is given below.  
 
\underline{\textbf{Predictive Learning-aided Control (\plc)}}: At time $t$,  do: 
\begin{achievements} 
\item (Estimation) Update $\bv{\pi}_{a}(t)$ with \ade($T_l, d, \epsilon_d$). 
\vspace{-.02in}
\item (Learning) 
Solve the following empirical  problem and compute the optimal Lagrange multiplier $\bv{\gamma}^*(t)$, i.e.,\footnote{The dual problem (\ref{eq:dualproblem-empirical}) is always concave  \cite{bertsekasoptbook}. Thus, it can be solved efficiently with existing convex optimization methods. 
Moreover, our results can be generalized to the case when the optimal Lagrange multiplier is computed only approximately. In particular, our results can be similarly proven if we obtain $\tilde{\bv{\gamma}}^*(t) = \bv{\gamma}^*(t)+err$. In this case, the utility bound remains the same and the delay bound will increase by $O(err)$.}
\begin{eqnarray}
\max:\,\,\, g(\bv{\gamma}, \bv{\pi}_{a}(t)),\quad \text{s.t.}\,\,\, \bv{\gamma}\succeq\bv{0},\label{eq:dualproblem-empirical}
\end{eqnarray}
If $\bv{\gamma}^*(t)=\infty$, set $\bv{\gamma}^*(t)=V\log(V)\cdot\bv{1}$. 
If $W_m(t-1) = T_l$ and $\bv{\pi}_{a}(t)\neq \bv{\pi}_{a}(t-1)$, set $\bv{q}(t+w+1)=\bv{0}$, i.e., drop all packets currently in the queues. 


\vspace{-.02in} 
\item (Control) At every time slot $t$,  observe the current network state $S(t)$ and the backlog $\bv{q}(t)$. If $S(t)=s_i$, choose $x^{(s_i)}\in\mathcal{X}_i$ that solves the following:
\begin{eqnarray}
\hspace{-.4in}&&\max: \quad -Vf(s_i, x)+\sum_{j=1}^{r}Q_j(t)\big[\mu_j(s_i, x)-A_j(s_i, x)\big]\nonumber\\
\hspace{-.4in}&&\quad\text{s.t.} \quad x\in\mathcal{X}_i,\label{eq:QLAeq}
\end{eqnarray}
where $Q_j(t) \triangleq q_j(t) + (\gamma_j^*(t) - \theta)^+$. Then, update the queues according to (\ref{eq:queuedynamic}) with Last-In-First-Out. $\Diamond$
\end{achievements} 

For readers who are familiar with the Backpressure (\bp) algorithm, e.g.,  \cite{neelynowbook} and \cite{huangneelypricing-ton},    the control component of \plc{} is  the \bp{} algorithm with its queue vector augmented by the empirical multiplier $\bv{\gamma}^*(t)$. 
Also note that packet dropping is introduced to enable quick adaptation to new dynamics if there is a distribution change. It occurs only when a   long-lasting distribution ends, which avoids dropping packets frequently in a fast-changing environment. 

We have the following remarks. 
\textbf{(i) Prediction usage:} Prediction is explicitly incorporated into  control by forming an average distribution and converting the distribution estimate into a Lagrange multiplier. 
The intuition for having $T_l=\max(V^c, e_w^{-2})$ is that when $e_w$ is small, we should rely on prediction as much as possible, and only switch to learned statistics when it is sufficiently accurate. 
\textbf{(ii) Connection with \rhc:} It is interesting to see that when $W_m(t)<T_l$, \plc{} mimics the commonly adopted receding-horizon-control method (\rhc), e.g., \cite{minghong-igcc-12}. The main difference is that, in \rhc{}, future \emph{states} are predicted and are directly fed into a predictive optimization formulation for computing the current action. 
Under \plc{},  \emph{distribution prediction} is  combined with \emph{historic state information} to  compute  an empirical multiplier for augmenting the controller. 
%
In this regard, \plc{} can be viewed as  exploring the benefits of statistics whenever it finds the system stationary (and does it automatically). 
\textbf{(iii)  Parameter selection:} The parameters in \plc{}  can be conveniently chosen as follows. 
First,  fix a detection error probability $\delta=V^{-\log(V)}$. Then,  choose a small $\epsilon_d$ and a  $d$ that satisfies  $d \geq 4\log(V)^2/\epsilon_d^2+w+1$. Finally, choose $T_l=\max(V^c, e_w^{-2})$ and $\theta$ according to (\ref{eq:theta-value}). 
%
 
While recent works \cite{huang-learning-sig-14} and \cite{huang-rlc-15} also design learning-based algorithms that utilize historic information, they do not consider  prediction  and do not provide insight on its benefits and the impact of prediction error. 
Moreover, \cite{huang-learning-sig-14} focuses on stationary systems and \cite{huang-rlc-15} adopts a frame-based scheme.  
 
%

%
%


\section{Performance Analysis} \label{section:performance} 
This section presents the performance results of \plc. We focus on four metrics, detection efficiency, network utility, service delay,  and algorithm convergence. The metrics are chosen to represent robustness, resource utilization efficiency, quality-of-service, and adaptability, respectively. 
Since our main objective  is to investigate how these metrics behave under our algorithm, we  treat network parameters $M$ and $r$ as constants.


%


\subsection{Detection and estimation}
We first look at the detection and estimation part. The following lemma summarizes the performance of \ade{}, which is affected by the prediction accuracy as expected.

\begin{lemma}\label{lemma:detection}
Under \ade($T_l, d, \epsilon_d$),  we have: 

(a) Suppose at a time $t$,  $\bv{\pi}(\tau_1)=\bv{\pi}_1$ for $\tau_1\in\mathcal{W}_{d}(t)$ 
and $\bv{\pi}(\tau_2)=\bv{\pi}_2\neq\bv{\pi}_1$ for all $\tau_2\in\mathcal{W}_{m}(t)$ and  $\max|\pi_{1i} - \pi_{2i}| >4(w+1)e_w/d$. 
Then, by choosing $\epsilon_d <\epsilon_0\triangleq \max|\pi_{1i} - \pi_{2i}|/2 - (w+1)e_w/d$ and $d>\ln\frac{4}{\delta}\cdot\frac{1}{2\epsilon_d^2} +w+1$, if $W_m(t)\geq W_d(t)=d$, with probability at least $1-\delta$, $b_m(t+1) = t+w+1$ and $\mathcal{W}_{m}(t+1) = \phi$, i.e., $W_m(t+1)=0$. 

(b) Suppose $\bv{\pi}(t)=\bv{\pi}$  $\forall\, t$. Then, if $W_{m}(t)\geq W_{d}(t)=d$, under \ade($T_l, d, \epsilon_d$) with $d\geq\ln\frac{4}{\delta}\cdot\frac{2}{\epsilon_d^2}+w+1$, with probability at least $1-\delta - (w+1)MT_l^{-2\log(T_l)}$, $b_m(t+1) = b_m(t)$. $\Diamond$   
\end{lemma}
\begin{proof}
See Appendix A. 
\end{proof}

Lemma \ref{lemma:detection} shows that for a stationary system, i.e., $\bv{\pi}(t)=\bv{\pi}$, $W_m(t)$ will likely grow to a large value (Part (b)), in which case $\bv{\pi}_{a}(t)$ will stay close to $\bv{\pi}$ most of the time. 
If instead $\mathcal{W}_m(t)$ and $\mathcal{W}_d(t)$ contain samples  from different distributions, \ade{} will reset $\mathcal{W}_m(t)$ with high probability. 
Note that since the first $w+1$ slots belong to prediction, \plc{} detects changes $O(w)$ slots faster compared to methods without prediction. 
The condition $\max|\pi_{1i} - \pi_{2i}| > 4(w+1)e_w/d$ can be understood as follows. If we want to distinguish two different distributions, we want the detection threshold to be no more than half of the distribution distance. Now with prediction, we want the potential prediction error to be no more than half of the threshold, hence the factor $4$. 
%
Also note that the delay involved in detecting a distribution change is nearly order-optimal, i.e.,  $d=O(1/\min_i|\pi_{1i} - \pi_{2i}|^2)$ time, which is known to be necessary for distinguishing two distributions   \cite{robin-lai-1985}. 
Moreover, $d=O(\ln(1/\delta))$ shows that a logarithmic window size is enough to ensure a high detection accuracy. 


\subsection{Utility and delay}
In this section, we look at the utility and delay performance of \plc. 
To state our results, we first define the following structural property of the  system. 
\begin{definition}
A system is called polyhedral with parameter $\rho>0$ under distribution $\bv{\pi}$ if the dual function $g(\bv{\gamma}, \bv{\pi})$ satisfies: 
\begin{eqnarray}
g(\bv{\gamma}^*, \bv{\pi})\geq g(\bv{\gamma}, \bv{\pi})+\rho\|\bv{\gamma}^*_{\bv{\pi}}-\bv{\gamma}\|.\,\,\,\Diamond\label{eq:polyhedral}
\end{eqnarray}
\end{definition}

The polyhedral property often holds for practical systems, especially when  action sets are finite (see \cite{huangneely_dr_tac} for more discussions). 
%
%

\subsubsection{Stationary system} 
We first consider stationary systems, i.e., $\bv{\pi}(t)=\bv{\pi}$. 
%
Our   theorem shows that \plc{} achieves the near-optimal utility-delay tradeoff for stationary networks. 
This result is important, as any good adaptive algorithm must be able to handle stationary settings well. 
%


\begin{theorem}\label{theorem:plc-stationary} 
Suppose $\bv{\pi}(t)=\bv{\pi}$, the system is polyhedral with $\rho=\Theta(1)$,  $e_w>0$, and $\bv{q}(0)=\bv{0}$. Choose $0<\epsilon_d<\epsilon_0\triangleq 2(w+1)e_w/d$,  $d=\log(V)^3/\epsilon_d^2$, $T_l= \max(V^c, e_w^{-2})$ for $c\in(0, 1)$ and 
\begin{eqnarray}
\theta=  2\log(V)^2(  1+\frac{V}{\sqrt{T_l}}).\label{eq:theta-value}
\end{eqnarray} 
Then,  with a sufficiently large $V$,  \plc{} achieves the following:  

(a) Utility: $f^{\plc}_{\text{av}} = f^{\bv{\pi}}_{\text{av}}+O(1/V)$

(b)  Delay: For all but an $O(\frac{1}{V})$ 
fraction of traffic, the average packet delay is $D = O(\log(V)^2)$ 

(c)  Dropping: The packet dropping rate is $O(V^{-1})$. $\Diamond$
\end{theorem}
\begin{proof}
See Appendix B. 
\end{proof}

Here $c$ is a constant used for deciding the learning time $T_l$. Choosing $\epsilon=1/V$, we see that \plc{} achieves the near-optimal $[O(\epsilon), O(\log(1/\epsilon)^2)]$ utility-delay tradeoff. Moreover,   prediction enables \plc{} to also greatly reduce the queue size (see Part (b) of Theorem \ref{theorem:plc-nonstationary}). 
%
%
Our result is different from the results  in \cite{zhang-predict-mm1-14} and  \cite{huang-predictive-14} for proactive service settings,  where delay vanishes as prediction power increases. This is because we only assume observability of future states but not pre-service, and highlights the difference between pre-service and pure prediction. 
Note that the performance of \plc{} does not depend heavily on $\epsilon_d$  in Theorem \ref{theorem:plc-stationary}. The value $\epsilon_d$  is more crucial for non-stationary systems, where a low false-negative rate is critical for performance. 
%
Also note that although packet dropping can occur during operation, the fraction of packets dropped is very small, and the resulting performance guarantee cannot be obtained by simply dropping the same amount of packets, in which case the delay will still be $\Theta(1/\epsilon)$. 

Although Theorem \ref{theorem:plc-stationary} has a similar form as those in  \cite{huang-rlc-15} and \cite{huang-learning-sig-14}, the analysis is very different, in that (i) prediction error must be taken into account, and (ii) \plc{} performs sequential detection and decision-making. 



\subsubsection{Piecewise stationary system} 
We now turn to the non-stationary case and consider the scenario where $\bv{\pi}(t)$ changes over time. In this case, we see that prediction is critical as it significantly accelerates convergence and helps to achieve good performance  when each distribution only lasts for a finite time. 
As we know that when the distribution can change arbitrarily, it is hard to even define optimality. 
Thus, we consider the case when the system is piecewise stationary, i.e., each distribution lasts for a duration of time, and study how the algorithm optimizes the performance for each distribution. 
%

The following theorem summarizes the performance of \plc{} in this case. In the theorem, we define $D_k\triangleq t_k+d-t^*$, where $t^*\triangleq \sup\{t<t_k+d: t\,\, \text{is a reset point}\}$, i.e., the most recent  time when  a cycle with size no smaller than $T_l$ ends (recall that reset points are marked in step (ii) of \ade{} and $d\geq w+1$). 

\begin{theorem}\label{theorem:plc-nonstationary} 
Suppose $d_k\geq 4d$ and the system is polyhedral with $\rho=\Theta(1)$   for all $k$,. Also, suppose there exists $\epsilon^*_0=\Theta(1)>0$ such that $\epsilon^*_0\leq\inf_{k, i}|\pi_{ki} - \pi'_{k-1i}|$ and $\bv{q}(0)=\bv{0}$. 
%
Choose $\epsilon_d<\epsilon_0^*$  in \ade{}, and choose $d$, $\theta$ and $T_l$ as in Theorem \ref{theorem:plc-stationary}. Fix any distribution $\bv{\pi}_k$ with length $d_k= \Theta(V^{1+a}T_l)$ for some $a=\Theta(1)>0$.\footnote{The constant $a$ here is introduced to show that our results hold as long as $d_k$ is larger than $O(VT_l)$.} 
Then, under \plc{} with a sufficiently large $V$,   if $\script{W}_m(t_k)$ only contains samples after $t_{k-1}$, we achieve the following with probability $1- O(V^{-3\log(V)/4})$:  

(a) Utility:  $f^{\plc}_{\text{av}} = f^{\bv{\pi}_k}_{\text{av}}+O(1/V) + O(\frac{D_k\log(V)}{T_lV^{1+a}})$ 

(b) Queueing:  $\overline{q}_{av} = O((\min(V^{1-c/2}, Ve_w)+1)\log^2(V) + D_k+d)$. 

(c) In particular, if $d_{k-1}=\Theta(T_lV^{a_1})$ for $a_1=\Theta(1)>0$ and $W_m(t_{k-1})$ only contains samples after $t_{k-2}$, then with probability $1-O(V^{-2})$,  $D_k = O(d)$,  $f^{\plc}_{\text{av}} = f^{\bv{\pi}_k}_{\text{av}}+O(1/V)$ and $\overline{q}_{av}=O(\min(V^{1-c/2}, Ve_w)+\log^2(V))$. 
$\Diamond$ 

\end{theorem}
\begin{proof}
See Appendix C. 
\end{proof}

A few remarks are in place. (i) Theorem \ref{theorem:plc-nonstationary}  shows that, with an increasing prediction power, i.e., a smaller $e_w$, it is possible to simultaneously reduce network queue size and the time it takes to achieve a desired average performance (even if we do not execute actions ahead of time). 
The requirement $d_k= \Theta(V^{1+a}T_l)$ can be strictly less than the $O(V^{2-c/2+a})$ requirement for \rlc{} in \cite{huang-rlc-15} and the $O(V^2)$ requirement of \bp{} for achieving the same average utility. This implies that \plc{}  finds a good system operating point faster, a desirable feature for network  algorithms. 
(ii) The dependency on $D_k$ is necessary. This is because \plc{} does not perform packet dropping if  previous intervals do not exceed length $T_l$. 
As a result, the accumulated backlog can affect decision making in the current interval. Fortunately the queues are shown to be small and do not heavily affect performance (also see simulations).  
(iii) To appreciate the queueing result, note that \bp{} (without learning) under the same setting will result in an $O(V)$ queue size. 
%
%
%

Compared to the analysis in \cite{huang-rlc-15}, one complicating factor in proving Theorem \ref{theorem:plc-nonstationary}  is that \ade{} may not always throw away samples from a previous interval. 
%
%
Instead, \ade{} ensures that with high probability, only $o(d)$ samples from a previous interval will remain. This ensures high learning accuracy and fast convergence of \plc{}. 
One interesting special case not covered in the last two theorems is when $e_w=0$. In this case,  prediction is perfect and $T_l=\infty$, and \plc{} always runs with $\bv{\pi}_a(t)=\frac{1}{w+1}\sum_{k=0}^{w}\hat{\bv{\pi}}(t+k)$, which is the exact average distribution. 
For this case, we have the following result. 

\begin{theorem}\label{theorem:perfect-plc}  
Suppose $e_w=0$ and $\bv{q}(0)=\bv{0}$. Then, \plc{} achieves the following: 

(a) Suppose $\bv{\pi}(t)=\bv{\pi}$ and the system is polyhedral with $\rho=\Theta(1)$. Then, under the conditions of Theorem  \ref{theorem:plc-stationary}, \plc{} achieves the $[O(\epsilon), O(\log(1/\epsilon)^2)]$ utility-delay tradeoff. 

(b) Suppose $d_k\geq d\log^2(V)$ and the system is polyhedral with $\rho=\Theta(1)$ under each $\bv{\pi}_k$. Under the  conditions of Theorem \ref{theorem:plc-nonstationary},  for an interval $d_k\geq V^{1+\epsilon}$ for any $\epsilon>0$, \plc{} achieves that  $f^{\plc}_{\text{av}} = f^{\bv{\pi}_k}_{\text{av}}+O(1/V)$ and $\expect{\bv{q}(t_k)}=O(\log^4(V))$. $\Diamond$ 
\end{theorem} 
\begin{proof}
See Appendix D. 
\end{proof}

The intuition  is that since prediction is perfect, i.e., $\bv{\pi}_a(t)=\bv{\pi}_k$ during $[t_k+d, t_{k+1}-w]$. Therefore, a better performance can be achieved. 
The key challenge in  this case is that \plc{} does not perform any packet dropping. Thus, queues can build up and one needs to show that the queues will be concentrating around $\theta\cdot\bv{1}$ even when the distribution changes.

\subsection{Convergence time} 
We now consider the algorithm  convergence time, which is an important evaluation metric and measures how long it takes for an algorithm to reach its steady-state. 
While recent works \cite{huang-rlc-15}, \cite{huang-learning-sig-14}, \cite{neely-convergence-16}, and \cite{heavy-ball-sig16} also investigate algorithm convergence time, they do not consider utilizing prediction in learning and do not study the impact of prediction error. 

To  formally state our results, we adopt the following definition of convergence time from \cite{huang-learning-sig-14}.

\begin{definition} 
Let $\zeta>0$ be a given constant and let $\bv{\pi}$ be a system distribution.  The $\zeta$-convergence time of a control algorithm, denoted by $T_{\zeta}$, is the time it takes for the effective queue vector $\bv{Q}(t)$ to get to within $\zeta$ distance of $\bv{\gamma}^*_{\bv{\pi}}$, i.e.,\footnote{$T_{\zeta}$ is essentially the \emph{hitting time} of the process $\bv{Q}(t)$ to the area $\{||\bv{Q}(t)-\bv{\gamma}^*_{\bv{\pi}}||\leq\zeta\}$ \cite{aldousmarkovchainbook}. 
}
\begin{eqnarray}
T_{\zeta}\triangleq\inf\{t\,|\, ||\bv{Q}(t)-\bv{\gamma}^*_{\bv{\pi}}||\leq\zeta\}.\,\,\,\Diamond\label{eq:con-time}
\end{eqnarray}
\end{definition}
With this definition, we have the following theorem. Recall that $w\leq d=\Theta(\log(V)^2)$. 
\begin{theorem}\label{theorem:plc-conv-stationary} 
Assuming all conditions in Theorem \ref{theorem:plc-nonstationary}, except that $\bv{\pi}(t)=\bv{\pi}_k$ for all $t\geq t_k$. If $e_w=0$, 
under \plc{}, 
\begin{eqnarray}
\hspace{-.2in}\expect{T_{G}} &=& O(\log^4(V)). \label{eq:convergence-time0}
\end{eqnarray}
Else suppose $e_w>0$. Under the  conditions of Theorem \ref{theorem:plc-nonstationary},  with probability $1-O(\frac{1}{VT_l} +\frac{D_k}{V^2T_l})$, 
\begin{eqnarray}
\hspace{-.2in}\expect{T_G}&=&O(\theta+T_l+ D_k+w)  \label{eq:convergence-time1} \\ 
\hspace{-.2in}\expect{T_{G_1}} &=& O(d).  \label{eq:convergence-time2}
\end{eqnarray} 
Here $G=\Theta(1)$ and  $G_1=\Theta(D_k+2\log(V)^2(1+Ve_w))$, where $D_k$ is defined in Theorem \ref{theorem:plc-nonstationary} as the most recent reset point before $t_k$. In particular, if $d_{k-1}=\Theta(T_lV^{a_1})$ for some $a_1=\Theta(1)>0$ and $\theta=O(\log(V)^2)$, then with probability $1-O(V^{-2})$,  $D_k = O(d)$, and $\expect{T_{G_1}} = O(\log^2(V))$. 
$\Diamond$ 
\end{theorem}
\begin{proof}
See Appendix E. 
\end{proof}

The assumption $\bv{\pi}(t)=\bv{\pi}_k$ for all $t\geq t_k$ is made to avoid the need for specifying the length of the intervals. 
%
%
It is interesting to compare (\ref{eq:convergence-time0}), (\ref{eq:convergence-time1}) and (\ref{eq:convergence-time2}) with the convergence results in  \cite{huang-learning-sig-14} and \cite{huang-rlc-15} without prediction, where it was shown that the convergence time is $O(V^{1-c/2}\log(V)^2+V^c)$, with a minimum of $O(V^{2/3})$. 
Here although we may still need $O(V^{2/3})$ time for getting into an $G$-neighborhood (depending on $e_w$), getting to the $G_1$-neighborhood can take only an $O(\log^2(V))$ time, which is much faster compared to previous results, e.g., when $e_w=o(V^{-2})$ and $D_k=O(w)$, we have $G_1 = O(\log^2(V))$. 
This confirms our intuition that \emph{prediction accelerates algorithm convergence} and demonstrates the power of  (even imperfect) prediction. 
%

\section{Simulation}\label{section:sim} 
In this section, we present simulation results of \plc{} in a two-queue system shown in Fig. \ref{fig:2q}.  
Though being simple, the system  models various settings, e.g., a two-user downlink transmission problem in a mobile network,  a CPU scheduling problem with two applications, or  an inventory control system where two types of orders are being processed. 
\begin{figure}[ht]
\begin{center}
\includegraphics[width=1.8in, height=0.8in]{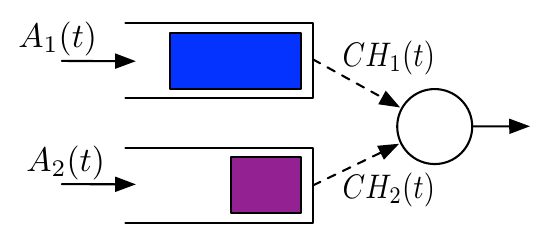}
\vspace{-.1in}
\caption{A single-server two-queue system. Each queue receives random arrivals. The server can only serve one queue at a time.} 
\label{fig:2q}
\end{center}
\vspace{-.2in}
\end{figure}

$A_j(t)$ denotes the number of arriving packets to queue $j$ at time $t$. We assume $A_j(t)$ is i.i.d.  being $1$ or $0$ with probabilities $p_j$ and $1-p_j$, and use $p_1=0.3$ and $p_2=0.6$. Thus, $\lambda_1=0.3$ and $\lambda_2=0.6$. 
Each queue has a time-varying channel condition. We denote $CH_j(t)$  the channel condition of queue $j$ at time $t$. We assume that $CH_j(t)\in \mathcal{CH}_j$ with $\mathcal{CH}_1= \{0, 1\}$ and $\mathcal{CH}_2=\{1, 2\}$. The channel distributions are assumed to be uniform. 
At each time, the server determines the power allocation to each queue. We use $P_j(t)$ to denote  the power allocated to queue $j$ at time $t$. 
Then, the instantaneous service rate $q_j(t)$ gets is given by:
\begin{eqnarray}
\mu_j(t) = \log(1+CH_j(t) P_j(t)). 
\end{eqnarray} 
We assume that $P_j(t)\in\mathcal{P}=\{ 0, 1, 2\}$ for $j=1, 2$, and at each time only one queue can be served. 
The objective is to stabilize the system with minimum average power. 
It can be verified that Assumptions \ref{assumption:bdd-LM} and \ref{assumption:unique} both  hold in this example. 

We compare \plc{} with \bp{} in two cases. 
 The first case is a stationary system where the arrival distributions remain constant. 
 The second case is a non-stationary case, where we change the arrival distributions during the simulation.  
 In both cases we simulate the system for $T=5\times10^4$ slots with $V\in\{20, 50, 100, 150, 200, 300\}$. We set $w+1=5$ and generate prediction error  by adding uniform random noise to distributions with max value $e(k)$ (specified below). 
We also use $\epsilon_d=0.1$, $\delta=0.005$ and  $d=2\ln(4/\delta)/\epsilon^2+w+1$. We also simplify the choice of $\theta$ and  set it to $\theta=\log(V)^2$. 
 
We first examine the long-term performance. Fig. \ref{fig:utility-delay-stationary} shows the utility-delay performance of \plc{} compared to \bp{} in the stationary setting. There are two \plc{} we simulated, one is with $e_w=0$ (\plc) and the other is with $e_w=0.04$ (\plc-e). From the plot, we see that both \plc{}s achieve a similar utility as \bp{}, but guarantee a much smaller delay. The reason \plc-e has a better performance is due to packet dropping. We observe around an average packet dropping rate of $0.06$. As noted before, the delay reduction of \plc{} cannot be achieved by simply dropping this amount of packets. 
 \begin{figure}[ht]
\begin{center}
\vspace{-.1in}
\includegraphics[width=3.3in, height=1.6in]{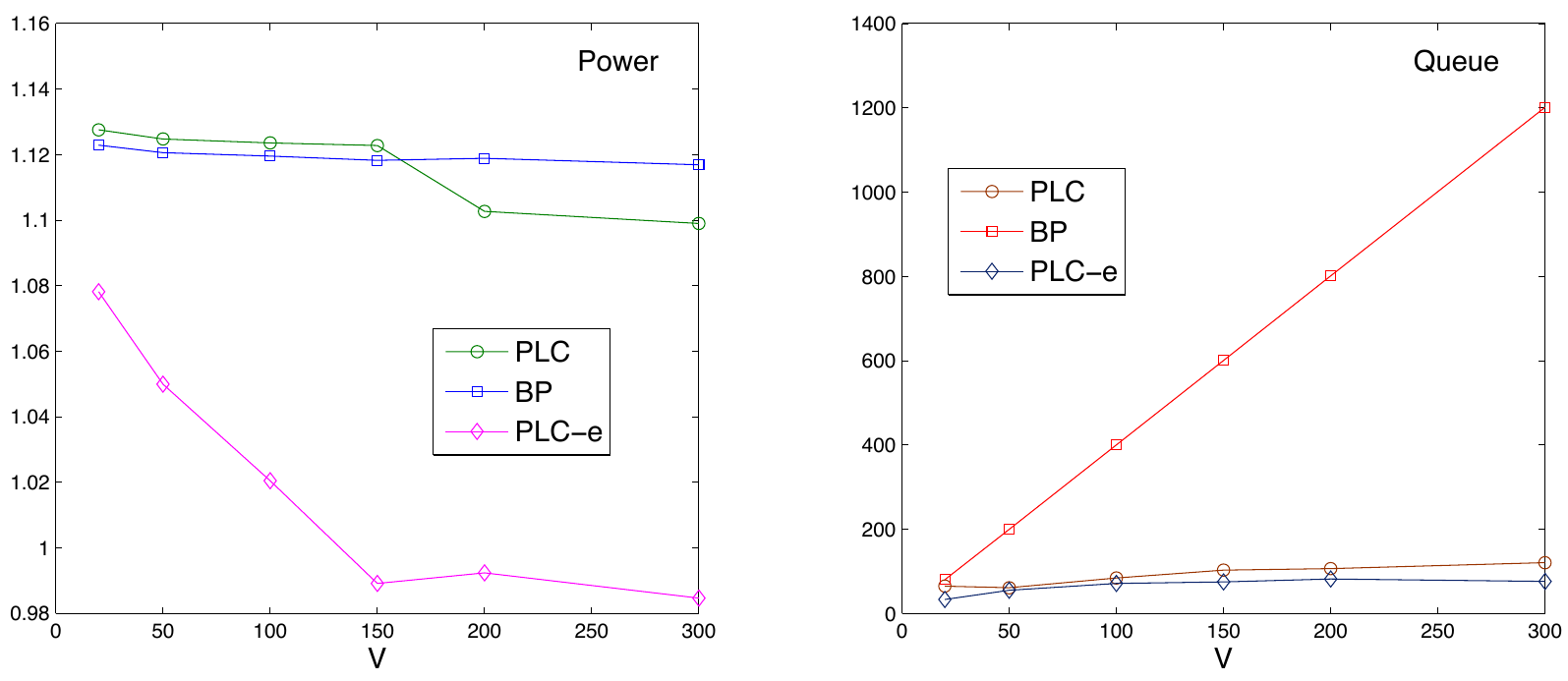}
\vspace{-.15in}
\caption{Utility and delay performance comparison between \plc{} and \bp. }  
\label{fig:utility-delay-stationary}
\end{center}
\vspace{-.1in}
\end{figure}

%

Next, we take a look at the detection and convergence performance of \plc{}. Fig. \ref{fig:utility-delay-nonstationary} shows the performance of \plc{} with perfect prediction ($e_w=0$), \plc{} with  prediction error ($e_w=0.04$) and \bp{} when the underlying distribution changes. Specifically, we run the simulation for $T=5000$ slots and start with the arrival rates of $p_1=0.2$ and $p_2=0.4$. Then, we change them to  $p_1=0.3$ and $p_2=0.6$ at time $T/2$. 
%
%
\begin{figure}[ht]
\begin{center}
\vspace{-.05in}
\includegraphics[width=3.3in, height=1.7in]{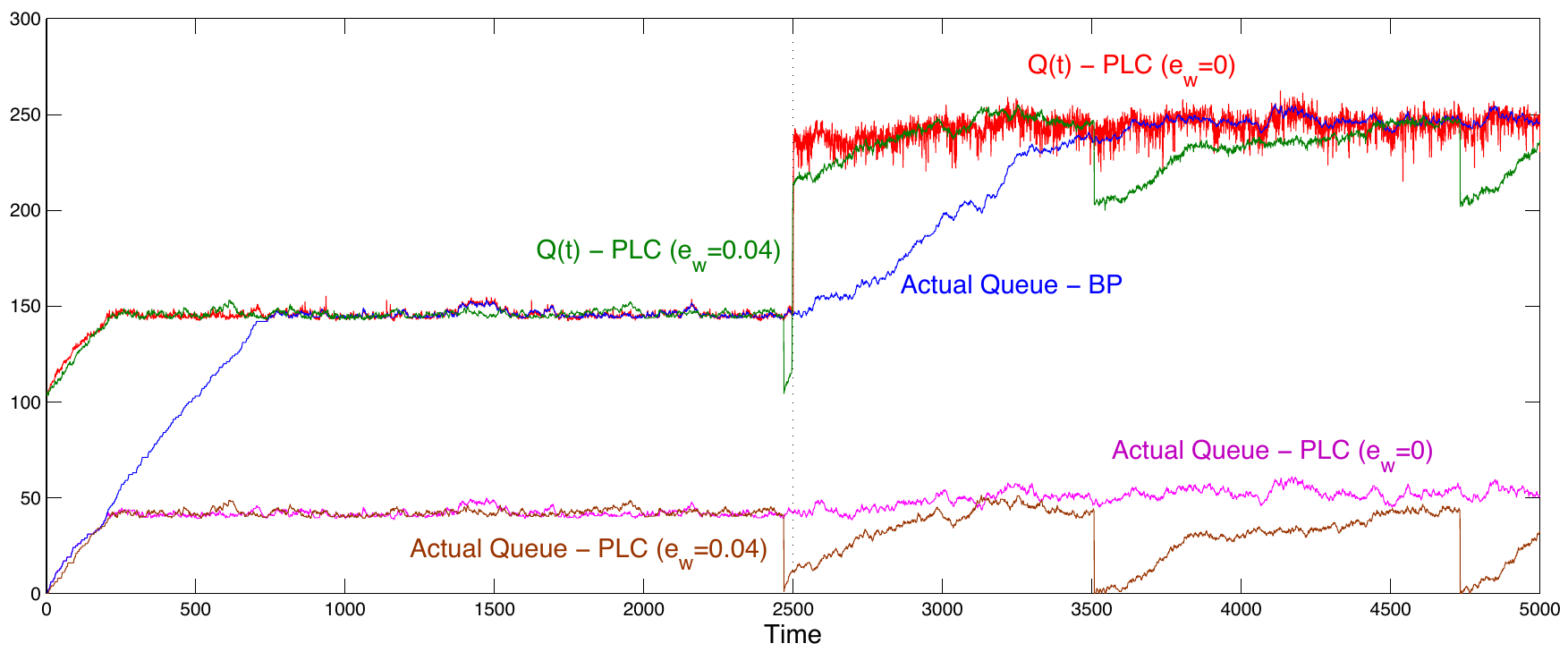}
\vspace{-.15in}
\caption{Convergence comparison between \plc{} and \bp{} for queue $1$ under $V=100$.  \plc{} ($e_w=0$) is the perfect case and \plc{} ($e_w=0.04$) contains prediction error. Both versions converge much faster compared to \bp.} 
\label{fig:utility-delay-nonstationary}
\end{center}
\vspace{-.2in}
\end{figure}

We can see from the green and red curves that \plc{} quickly adapts to the change and modifies the Lagrange multiplier accordingly. By doing so, the actual queues under \plc{}  (the purple and the brown curves) remain largely unaffected. For comparison, we see that \bp{} takes a longer time to adapt to the new distribution and results in a larger queue size. 
We also see that during the $5000$ slots, \plc{} ($e_w=0.04$) drops packets $3$ times (zero for the first half), validating the results in Lemma \ref{lemma:detection} and Theorem \ref{theorem:plc-stationary}. Moreover, after the distribution change, \plc{} ($e_w=0.04$)  quickly adapts to the new equilibrium, despite having imperfect prediction. 
The fast convergence result also validates our theorem about short term utility performance under \plc{}. Indeed, if we look at slots during time $200-500$, and slots between $2500-3500$, we see that when  \bp{} is learning the target backlog, \plc{} already operates near the optimal mode. This  shows the benefits of prediction and learning in stochastic network control.  

%
%
%
%
 
\section{Conclusion}\label{section:conclusion}
We investigate the problem of stochastic network optimization in the presence of imperfect state prediction and non-stationarity. 
Based on a novel distribution-accuracy curve prediction model, we develop the  \emph{predictive learning-aided control} (\plc) algorithm. \plc{} is an online algorithm that 
contains three main functionalities, sequential distribution estimation and change detection, dual learning, and online  queue-based control. 
We show that \plc{} simultaneously achieves good long-term performance, short-term queue size reduction, accurate change detection, and fast algorithm convergence. 
Our results demonstrate that state prediction  can help improve performance  and quantify the benefits of prediction in  four important    metrics, i.e., utility (efficiency), delay (quality-of-service), detection (robustness), and convergence (adaptability). They provide new insight for joint prediction, learning and optimization in stochastic networks.

\bibliographystyle{unsrt}
\bibliography{mybib}

\section*{Appendix A - Proof of Lemma \ref{lemma:detection}}\label{section:appendix-A}
\textbf{(Proof of Lemma \ref{lemma:detection})} 
We prove the performance of \ade($T_l, d, \epsilon$) with an argument inspired by  \cite{adwin07}.  We  make use of the following concentration result. 
\begin{theorem}\label{theorem:concentration} \cite{chung_concentration} 
Let $X_1$, ..., $X_n$ be independent random variables with $\prob{X_i=1}=p_i$, and $\prob{X_i=0}=1-p_i$. Consider $X=\sum_{i=1}^nX_i$ with expectation $\expect{X}=\sum_{i=1}^np_i$. Then, we have:
\begin{eqnarray}
\prob{X\leq \expect{X}-m} &\leq& e^{\frac{-m^2}{2\expectm{X}}}, \label{eq:low-tail}\\
\prob{X\geq \expect{X}+m} &\leq& e^{\frac{-m^2}{2(\expectm{X}  +m/3 )}}. \quad\Diamond\label{eq:high-tail}
\end{eqnarray}
\end{theorem}
\begin{proof} (Lemma \ref{lemma:detection})
(\textbf{Part (a)}) In this case, it suffices to check condition (i) in \ade. Define 
\begin{eqnarray*}
 \tilde{\pi}_i^d(t) \triangleq\frac{1}{d} \bigg(\sum_{\tau=(t+w-d)_+}^{t-1}1_{[S(\tau) = s_i]} + \sum_{\tau\in \mathcal{W}_w(t)} \pi_i(\tau) \bigg),  
\end{eqnarray*}
i.e., $\tilde{\pi}_i^d(t)$ is defined with the true distributions in $\mathcal{W}_d(t)$. Denote $\epsilon_1 = (w+1)e_w/d$, we see then $\| \tilde{\bv{\pi}}^d(t) - \hat{\bv{\pi}}^d(t)\|\leq \epsilon_1$.  
Thus, for any $\epsilon> 0$, we have: 
\begin{eqnarray}
&&\prob{\|\hat{\bv{\pi}}^d(t) - \hat{\bv{\pi}}^m(t)\|_{TV} \leq \epsilon}\nonumber\\
  &\leq&\prob{\|\tilde{\bv{\pi}}^d(t) - \hat{\bv{\pi}}^m(t)\|_{TV} \leq \epsilon+\epsilon_1} \nonumber\\ 
 &\leq&  \prob{ |\tilde{\pi}^d_i(t) - \hat{\pi}^m_i(t)| \leq \epsilon+\epsilon_1}.  \label{eq:fooo-1}
\end{eqnarray} 
Choose $\epsilon=\frac{1}{2}\max|\pi_{1i} - \pi_{2i}| - 2\epsilon_1>0$ and let $\epsilon_0 = \epsilon+\epsilon_1$. 
Fix $\alpha\in(0, 1)$ and consider   $i\in\arg\max_i|\pi_{1i} - \pi_{2i}|$. We have: 
\begin{eqnarray} 
\hspace{-.2in}&& \prob{|\tilde{\pi}^d_i(t) - \hat{\pi}^m_i(t)| \leq   \epsilon_0}\nonumber\\
\hspace{-.2in}&&\qquad \leq \prob{ \{|\tilde{\pi}^d_i(t) - \pi_{1i}| \geq \alpha   \epsilon_0\}\nonumber\\
\hspace{-.2in}&&\qquad\qquad\qquad\cup\{|\hat{\pi}^m_i(t) - \pi_{2i}| \geq (1-\alpha)  \epsilon_0\}   }\nonumber\\
\hspace{-.2in}&&\qquad\leq \prob{ |\tilde{\pi}^d_i(t) - \pi_{1i}| \geq \alpha \epsilon_0} \nonumber\\
\hspace{-.2in}&&\qquad\qquad\qquad+ \prob{ |\hat{\pi}^m_i(t) - \pi_{2i}| \geq (1-\alpha)  \epsilon_0}. \label{eq:prob-bdd}
\end{eqnarray}
Here the first inequality follows because if we have both $\{|\tilde{\pi}^d_i(t) - \pi_{1i}| <\alpha \epsilon_0\}$ and $\{|\hat{\pi}^m_i(t) - \pi_{2i}| < (1-\alpha)  \epsilon_0$, and $|\tilde{\pi}^d_i(t) - \hat{\pi}^m_i(t)| \leq   \epsilon_0$, we must have: 
\begin{eqnarray*}
\hspace{-.3in}&&\,\,|\pi_{1i}-\pi_{2i}|\leq |\tilde{\pi}^d_i(t) - \pi_{1i}|  +  |\hat{\pi}^m_i(t) - \pi_{2i}| + |\tilde{\pi}^d_i(t) - \hat{\pi}^m_i(t)| \\
\hspace{-.3in}&&\qquad\qquad\quad = 2\epsilon_0 <    |\pi_{1i} - \pi_{2i}|, 
\end{eqnarray*}
which contradicts the fact that $i$ achieves $\max_i|\pi_{1i} - \pi_{2i}|$. 
Using (\ref{eq:prob-bdd}) 
and Hoeffding inequality \cite{hoeffding1963}, we first have: 
\begin{eqnarray}
\hspace{-.4in}&&\prob{|\hat{\pi}^m_i(t) - \pi_{2i}| \geq (1-\alpha)\epsilon_0}\nonumber\\
\hspace{-.4in}&& \qquad\qquad \leq   2\exp(-2( (1-\alpha) \epsilon_0)^2W_{m}(t)). \label{eq:prob-bdd-derive1}
\end{eqnarray}
For the first term in (\ref{eq:prob-bdd}), we have: 
\begin{eqnarray}
\hspace{-.4in}&&\prob{|\tilde{\pi}^d_i(t) - \pi_{1i}| \geq \alpha\epsilon_0} \nonumber\\
\hspace{-.4in}&& \qquad\qquad   \leq   2\exp(-2(\alpha \epsilon_0)^2(W_{d}(t)-w-1)). \label{eq:prob-bdd-derive2}
\end{eqnarray}
 Equating the above two probabilities and setting the sum equal to $\delta$, we have $\alpha = \frac{\sqrt{W_{m}(t)/(W_{d}(t)-w-1)}}{1+\sqrt{W_{m}(t)/(W_{d}(t)-w-1)}}$, and 
\begin{eqnarray}
\epsilon_0 = \sqrt{\ln\frac{4}{\delta}} \cdot\frac{1 + \sqrt{ (W_{d}(t)-w-1)/W_{m}(t)}}{\sqrt{ 2(W_{d}(t)-w-1)}}. 
\end{eqnarray}
In order to detect the different distributions, we can choose $\epsilon_d<\epsilon_0$, which on the other hand requires that: 
\begin{eqnarray}
\hspace{-.4in}&&\epsilon_d\stackrel{(*)}{\leq} \sqrt{\ln\frac{4}{\delta}} \cdot \sqrt{ \frac{1}{2(d-w-1)} } < \epsilon_0\nonumber\\
\hspace{-.4in}&&\qquad\qquad\qquad\Rightarrow\, d> \ln\frac{4}{\delta}\cdot\frac{1}{2\epsilon_d^2} +w+1. 
\end{eqnarray}
Here (*) follows because $W_{d}(t)=d\leq W_{m}(t)$. 
This shows that whenever $W_{d}(t)=d\leq W_{m}(t)$ and the windows are loaded with non-coherent samples, error will be detected with probability $1-\delta$. 


%

(\textbf{Part (b)}) Note that for any time $t$,  the distribution will be declared changed if $||\hat{\bv{\pi}}^d(t) - \hat{\bv{\pi}}^m(t)||_{TV} > \epsilon_d$.  
Choose $\epsilon_d=2\epsilon_1$. 
Similar to the above, we have: 
\begin{eqnarray}
\hspace{-.2in}&&\prob{\|\hat{\bv{\pi}}^d(t) - \hat{\bv{\pi}}^m(t)\|_{TV} \geq \epsilon_d}\label{eq:prob-bdd2}\\
\hspace{-.2in}&& \qquad\leq\prob{\|\tilde{\bv{\pi}}^d(t) - \hat{\bv{\pi}}^m(t)\|_{TV} \geq \epsilon_d-\epsilon_1}\nonumber\\
\hspace{-.2in}&& \qquad\leq \prob{ \|\hat{\bv{\pi}}^d(t) -\bv{\pi}\|_{TV} \geq \alpha \epsilon_d/2} \nonumber\\
\hspace{-.2in}&&\qquad\qquad+ \prob{\|  \hat{\bv{\pi}}^m(t)  -\bv{\pi}\|_{TV} \geq (1-\alpha)\epsilon_d/2}. \nonumber
\end{eqnarray}
Using the same argument as in (\ref{eq:prob-bdd}), (\ref{eq:prob-bdd-derive1}) and (\ref{eq:prob-bdd-derive2}), we get: 
\begin{eqnarray}
\hspace{-.2in}&&\prob{\|\hat{\bv{\pi}}^d(t) - \hat{\bv{\pi}}^m(t)\|_{TV} \geq \epsilon_d} \leq \delta. \nonumber
\end{eqnarray}
This shows that  step (i) declares change with probability $\delta$. 

Next we show that step (ii) does not declare a distribution change with high probability. To do so, we use Theorem \ref{theorem:concentration} with $m=2\log(T_l)\sqrt{T_l}$ to have that, when $W_m(t)\geq T_l$, 
\begin{eqnarray*}
\prob{ \| \pi_i^m(t)  -\pi_i\| > \frac{2\log(T_l)}{\sqrt{T_l}} } \leq e^{-2\log(T_l)^2} = T_l^{-2\log(T_l)}. 
\end{eqnarray*}
Using the union bound, we get
\begin{eqnarray}
\prob{ \| \bv{\pi}^m(t)  -\bv{\pi}\| > \frac{2M\log(T_l)}{\sqrt{T_l}} } \leq  MT_l^{-2\log(T_l)}. 
\end{eqnarray} 
Thus, part (b) follows from the union bound over $k$. 
\end{proof}

\section*{Appendix B - Proof of Theorem \ref{theorem:plc-stationary}}\label{section:appendix-B}
\textbf{(Proof of Theorem \ref{theorem:plc-stationary})} 
Here we prove the utility-delay performance of \plc{} for a stationary system. We sometimes omit the $\bv{\pi}$ when it is clear. 
For our  analysis, define:
\vspace{-.06in}
\begin{eqnarray}
\hspace{-.3in}&&g_{s_i}(\bv{\gamma})=\inf_{x^{(s_i)}\in \mathcal{X}_i} \bigg\{Vf(s_i, x^{(s_i)})\label{eq:dual_single}\\
\hspace{-.3in}&&\qquad\qquad\qquad\qquad+\sum_j\gamma_j\big[A_j(s_i, x^{(s_i)})- \mu_j(s_i, x^{(s_i)})\big]\bigg\},\nonumber
\end{eqnarray}
to be the dual function when there is only a single state $s_i$. It is clear from equations (\ref{eq:dual_separable}) and (\ref{eq:dual_single}) that:
\begin{eqnarray}
g(\bv{\gamma}) = \sum_{i}\pi_{i}g_{s_i}(\bv{\gamma}). \label{eq:sum-dual}
\end{eqnarray}
We will also make use of the following results. 
\begin{lemma}\label{lemma:prob_multi_con} \cite{huangneely_dr_tac}  
Suppose the conditions in Theorem \ref{theorem:plc-stationary}  hold. Then, under \plc{} with $\bv{Q}(t)=\bv{q}(t)$, there exist constants $G, \eta_1=\Theta(1)$, i.e., both independent of $V$, such that whenever $\|\bv{q}(t) - \bv{\gamma}^*\|>G$, 
\begin{eqnarray}
\mathbb{E}_{\bv{\pi}} \{  \|  \bv{q}(t+1) - \bv{\gamma}^*  \|  \left.|\right. \bv{q}(t)\}\leq \|\bv{q}(t) - \bv{\gamma}^* \| -\eta_1. \quad \Diamond 
\end{eqnarray}
\end{lemma} 

\begin{lemma} \label{lemma:estimate-error} \cite{huang-learning-sig-14}
If $\|\bv{\pi}_a(t) - \bv{\pi}\|_{TV} \leq \epsilon$ and (\ref{eq:slackness}) holds for $\bv{\pi}_a(t)$, then $\bv{\gamma}^*(t)$ satisfies: 
\begin{eqnarray}
\| \bv{\gamma}^*(t) - \bv{\gamma}^*\|\leq b_0V \epsilon,
 \label{eq:beta-rate-delta} 
\end{eqnarray}
where $b_0=\Theta(1)$. $\Diamond$
\end{lemma}

\begin{lemma} \label{lemma:queue-bdd-drift} \cite{neely-convergence-16}
Suppose $Z(t)$ is a real-value random process with initial value $z_0$ that satisfies: 
\begin{enumerate}
\item $|Z(t+1)-Z(t)|\leq Z_{\max}$ where $Z_{\max}>0$
\item $\expect{Z(t+1)-Z(t)\,|\,Z(t)}\leq z(t)$ where $z(t) = Z_{\max}$ when $Z(t)<Z_u$ and $z(t) = -\eta$ with $0\leq \eta\leq Z_{\max}$ when $Z(t)\geq Z_u$ for some constant $Z_u$. 
\end{enumerate}
Then, there exist constants $r_z=\Theta(1)$, $0<\rho_z<1$, and $D=\frac{(e^{r_zZ_{\max}} -\rho_z )e^{r_zZ_u}  }{1-\rho_z}$, such that for every slot $t$, 
\begin{eqnarray}
\expect{e^{r_zZ(t)}} \leq D+ (e^{r_zz_0}  - D)\rho_z^t. \quad \Diamond
\end{eqnarray}
\end{lemma}

%
%

Now we prove Theorem \ref{theorem:plc-stationary}. 
\begin{proof} (Theorem \ref{theorem:plc-stationary}) 
(\textbf{Part (a) - Utility}) Define a Lyapunov function $L(t)\triangleq\frac{1}{2}\sum_jq_j(t)^2$. Then,  define the one-slot Lyapunov drift $\Delta(t)\triangleq\expect{ L(t+1) - L(t)\,|\,\bv{q}(t) }$. Using the queueing dynamic equation (\ref{eq:queuedynamic}), we have: 
\begin{eqnarray}
\Delta(t) \leq B - \sum_jq_j(t)\expect{ \mu_j(t)  -A_j(t) \,|\,\bv{q}(t)  }. 
\end{eqnarray}
Here $B\triangleq r\delta_{\max}^2$, and  the expectation is taken over $\bv{\pi}$ and the potential randomness in action selection. 
Adding to both sides the term $V\expect{f(t) \,|\,\bv{q}(t) }$, we first obtain: 
\begin{eqnarray}
\hspace{-.2in}&&\Delta(t) + V\expect{f(t) \,|\,\bv{q}(t) } \leq B \label{eq:drift-v}\\
\hspace{-.2in}&&\qquad\qquad + \sum_j\expect{ Vf(t) - q_j(t) [\mu_j(t)  - A_j(t)] \,|\,\bv{q}(t) }. \nonumber
\end{eqnarray}
Now add to both sides the term $\Delta_1(t)\triangleq \expect{ (\gamma^*_j(t)-\theta)^+[\mu_j(t)  - A_j(t)] \,|\,\bv{q}(t) }$, we get: 
\begin{eqnarray}
\hspace{-.2in}&&\Delta(t) + V\expect{f(t) \,|\,\bv{q}(t) } +\Delta_1(t) \label{eq:drift-v2}\\
\hspace{-.2in}&&\qquad\leq B + \sum_j\expect{  Vf(t) + Q_j(t) [\mu_j(t)  - A_j(t)] \,|\,\bv{q}(t) }. \nonumber\\
\hspace{-.2in}&&\qquad = B + g(\bv{Q}(t)) \nonumber\\ 
\hspace{-.2in}&&\qquad \leq B + f^{\bv{\pi}}_{\text{av}}. 
\end{eqnarray} 
The last inequality holds as follows. Define a convexified dual function as: 
\begin{eqnarray*}
\hspace{-.3in}&&\tilde{g}_{s_i}(\bv{\gamma})=\inf_{x^{(s_i)}_k\in \mathcal{X}_i} \bigg\{\sum_ka^{(s_i)}_kVf(s_i, x^{(s_i)})\label{eq:dual_single_conv}\\
\hspace{-.3in}&& \qquad\qquad +\sum_ka^{(s_i)}_k\sum_j\gamma_j\big[A_j(s_i, x_k^{(s_i)})- \mu_j(s_i, x_k^{(s_i)})\big]\bigg\},\nonumber
\end{eqnarray*}
and $\tilde{g}(\bv{\gamma}) = \sum_{i}\pi_{i}\tilde{g}_{s_i}(\bv{\gamma})$. We have $\tilde{g}(\bv{\gamma}) = g(\bv{\gamma})$. To see this, suppose one $x_k^{(s_i)}$ maximizes the term $Vf(s_i, x_k^{(s_i)}) + \sum_j\gamma_j\big[A_j(s_i, x_k^{(s_i)})- \mu_j(s_i, x_k^{(s_i)})\big]$.  Then, the optimal choice of $\{a^{(s_i)}_k\}$ values is to set $a^{(s_i)}_k=1$ and the other values zero, in which case we obtain $\tilde{g}(\bv{\gamma}) = g(\bv{\gamma})$.  
From the definition of $\tilde{g}(\bv{\gamma})$, we can exactly view the probabilities $\{a_k^{(s_i)}\}$ as specifying a stationary and randomized policy that achieves the minimum for a given $\bv{\gamma}$. 
Hence, the value achieved must be no larger than any other such policies, including the ones that achieve the optimal utility subject to stability, which is guaranteed to exist \cite{neelynowbook}, resulting in $g(\bv{Q}(t)) \leq f^{\bv{\pi}}_{\text{av}}$. 

Taking an expectation over $\bv{q}(t)$, carrying out a telescoping sum from $t=0$ to $t=T-1$, and dividing both sides by $VT$, we obtain: 
\begin{eqnarray}
\frac{1}{T}\sum_{t=0}^{T-1} \expect{f(t) }  \leq f^{\bv{\pi}}_{\text{av}} +B/V - \frac{1}{VT}\sum_{t=0}^{T-1}\expect{\Delta_1(t)}. \label{eq:aug-drift-step0}
\end{eqnarray}
To prove the utility performance, it remains to show that the last term is $O(1)$ in the limit, i.e., 
\begin{eqnarray}
\lim_{T\rightarrow\infty}\frac{1}{T}\sum_{t=0}^{T-1}\sum_j\expect{(\gamma^*_j(t)-\theta)^+[\mu_j(t)  - A_j(t)]} =O(1). \label{eq:drift-1overV}
\end{eqnarray}

To prove (\ref{eq:drift-1overV}), consider the system evolution over the timeline. 
From the detection algorithm, we see that the timeline is divided into intervals separated by reset points. 
Moreover, since $\mathcal{W}_{m}(t)$ and $\mathcal{W}_{d}(t)$ are restarted, and $\bv{q}(t)$ is set to zero at reset points, 
  these intervals form renewal cycles with initial backlog $\bv{q}(t)=\bv{0}$ (see Fig. \ref{fig:timeline-structure}).
%
\begin{figure}[ht]
\begin{center}
\vspace{-.15in}
\includegraphics[width=3.3in, height=0.8in]{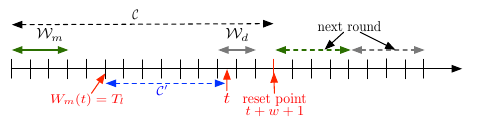}
\vspace{-.2in}
\caption{Timeline divided into intervals.}
\label{fig:timeline-structure}
\end{center}
\vspace{-.15in}
\end{figure}

Label the cycles by $\{\mathcal{C}_k, k=0, 1, ...\}$. We thus have:  
\begin{eqnarray}
\hspace{-.2in}&&\lim_{T\rightarrow\infty}\frac{1}{T}\sum_{t=0}^{T-1}\sum_j\expect{(\gamma^*_j(t)-\theta)^+[\mu_j(t)  - A_j(t)]} \label{eq:renewal-cycle} \\
\hspace{-.2in}&& =\frac{\expect{\text{cost}}}{\expect{\text{length}}}\triangleq\frac{\expect{\sum_{t\in \mathcal{C}_k}\sum_j  (\gamma^*_j(t)-\theta)^+[\mu_j(t)  - A_j(t)]   }  }{ \expect{|\mathcal{C}_k|}  }. \nonumber
\end{eqnarray} 
Below, we omit the index $k$ and use $d_{m}\triangleq\max_{t}\{b_d^s(t)-b_m(t)\}$  to denote the size of $\mathcal{C}$. 
%
Also, let $c_0$ be such that $e_w=\Theta(V^{-c_0/2})$, and write $T_l=V^{c_1}$ where $c_1\triangleq\max(c, c_0)$. Since $e_w>0$, we have $T_l<\infty$. 

We first show that the probability for having a small $d_{m}$ (w.r.t. $T_l$) is small. 
Denote the event $\mathcal{E}_1$ that $d_{m} \geq V^{2+c_1}$ and $W_m(t)=T_l$ at $t=T_l+d-w$ slots from the beginning of $\mathcal{C}$, i.e., step (i) of \ade{} does not declare any change before $W_m(t)=T_l$ occurs. 
%
Using Lemma \ref{lemma:detection}, $T_l\geq V^{c}$, and the fact that for a large $V$, $d=\log(V)^3/\epsilon_d^2\geq\frac{2}{\epsilon_d^2}\ln(4/\delta)+w+1$ for $\delta=V^{-\log(V)}$, we have that: 
\begin{eqnarray} 
 \prob{  \mathcal{E}_1^c} &\leq& \delta T_l+ V^{2+c_1}\cdot (\delta + (w+1)MT_l^{-2\log(T_l)}) \nonumber\\
&\leq&V^{-c^2\log(V)} \leq  V^{-3}. \label{eq:length-prob}
\end{eqnarray} 
Here $\mathcal{E}_1^c$ denotes the complementary event of $\mathcal{E}_1$. 
Therefore, with probability at least $1-V^{-3}$, $d_m\geq V^{2+c_1}$, which  implies that: 
\begin{eqnarray}
 \expect{|\mathcal{C}|} \geq V^{2+c_1}/2. \label{eq:length-exp}
\end{eqnarray}


Conditioning on $\mathcal{E}_1$, we see that \plc{} will compute the empirical multiplier with statistics in $\script{W}_m(t)$ after $t=T_l+d-w$ slots, and use it until a new change is declared (see Fig. \ref{fig:timeline-structure}). 
%
Denote this period of time, i.e., after the multiplier is computed until the cycle ends, by $\mathcal{C}'$ and its length  by $d'_m$ (see Fig. \ref{fig:timeline-structure}).  
%
%
We have $d'_m=d_m-O(V^{c_1}) -w-1 = \Omega(V^{2+c_1})$ time (the first $V^{c_1}$ slots are for learning and the last $O(w+1)$ slots are not contained in both $\mathcal{W}_m(t)$ and $\mathcal{W}_d(t)$ due to  \ade). 



Denote another event  $\mathcal{E}_2\triangleq\{ \|\hat{\bv{\pi}}^m(t) - \bv{\pi}\|\leq  \frac{4\log(V)}{V^{c_1/2}}\}$, where $t$ is when $W_m(t)=T_l$. That is, 
 the distribution $\hat{\bv{\pi}}^m(t)$ is close to the true distribution, for $t\in\mathcal{C}'$ (Note that $\hat{\bv{\pi}}^m(t)$ remains constant during $\mathcal{C}'$). 
Using Theorem \ref{theorem:concentration}, we have that: 
\begin{eqnarray}
\prob{\mathcal{E}_2} \geq 1-Me^{-4\log(V)^2}.\label{eq:event2-prob}
\end{eqnarray}
Thus, 
\begin{eqnarray}
\hspace{-.3in}&& \prob{\mathcal{E}_2 \,|\, \mathcal{E}_1} = \frac{ \prob{\mathcal{E}_2\cap\mathcal{E}_1}  }{\prob{\mathcal{E}_1}}
\geq  1-Me^{-4\log(V)^2} -V^{-3}  \nonumber \\
\hspace{-.3in}&&\qquad\qquad\qquad\qquad\qquad\quad\,\,\, \geq 1- 2V^{-3}. \label{eq:event3-prob}   
\end{eqnarray}

With $\mathcal{E}_1$ and $\mathcal{E}_2$, we now bound $\expect{ \text{cost} }$, where  $\text{cost} \triangleq\sum_{t\in \mathcal{C}}\sum_j  (\gamma^*_j(t)-\theta)^+[\mu_j(t)  - A_j(t)]$. 
First, when $\mathcal{E}_1^c$ takes place, we either have $d_{m} \leq V^{2+c_1}$, denoted by $\mathcal{E}_{1a}^c$, or $d_{m} \geq V^{2+c_1}$ but step (i) of \ade{} declares changes before $W_m(t)=T_l$, denoted by $\mathcal{E}_{1b}^c$. Given $\mathcal{E}_{1a}^c$, the cost is no more than $V\log(V)\delta_{\max} V^{2+c_1}$. For $\mathcal{E}_{1b}^c$, we first see that: 
\begin{eqnarray}
\hspace{-.2in}&& \prob{\mathcal{E}_{1b}^c}\nonumber\\
\hspace{-.2in}&=&\prob{ d_{m} \geq V^{2+c_1} \text{and} \nonumber\\
\hspace{-.2in}&&\quad \text{at least one  change declared in the first}\, T_l+d-w\, \text{slots}} \nonumber\\
\hspace{-.2in}&\leq& \delta (T_l+d-w) \leq V^{-\log(V)/2}. \label{eq:prob-e-1b-c}
\end{eqnarray}
Also, given $\mathcal{E}_{1b}^c$, if we denote the first time a change is declared by $T_{1b}$, we have: 
\begin{eqnarray}
 \expect{|\mathcal{C}| \,|\,\mathcal{E}_{1b}^c} &\leq&  T_{1b}+ \expect{|\mathcal{C}| \,|\,d_{m} \geq V^{2+c_1}-T_{1b}}  \nonumber\\
 &\leq& T_l+d+2\expect{|\mathcal{C}|} \leq 3 \expect{|\mathcal{C}|}. \label{eq:cond-cycle-foo}
\end{eqnarray} 
The first step follows because after the first declaration, the requirement for any additional declaration is removed, and the second step follows because  $T_{1b}\leq T_l+d-w$ and $\expect{|\mathcal{C}| \,|\,d_{m} \geq V^{2+c_1}-T_{1b}}\leq 2\expect{|\mathcal{C}|}$. 
Thus, 
\begin{eqnarray} 
 &&\expect{\text{cost} \,|\,\mathcal{E}_1^c} \label{eq:bound-ex-1}
\\ 
&& \qquad \leq V\log(V)\delta_{\max}\cdot (\frac{V^{2+c_1}}{V^{c^2\log(V)} } + \frac{3\expect{|\mathcal{C}|}}{V^{\log(V)/2}}). \nonumber 
\end{eqnarray}
Here we have used Lemma $1$ in \cite{huangneely_dr_tac} and the learning step in \plc{} that ensures  $\bv{\gamma}^*(t)=O(V\log(V))$. 
On the other hand, 
\begin{eqnarray}
\expect{\text{cost} \,|\,\mathcal{E}_1, \mathcal{E}_2^c}\leq V\log(V)\delta_{\max}\expect{|\mathcal{C}|\,|\,\mathcal{E}_1, \mathcal{E}_2^c}.  \label{eq:bound-ex-2}
\end{eqnarray}
Let us now bound $\expect{|\mathcal{C}|\,|\,\mathcal{E}_1, \mathcal{E}_2^c}$. 
Define $\script{E}_{2a}^c = \{ y\in(\sigma, 2\sigma]  \}$ and  $\script{E}_{2b}^c = \{ y>2\sigma  \}$, where $y\triangleq\|\hat{\bv{\pi}}^m(t) - \bv{\pi}\|$ and $\sigma\triangleq\frac{4\log(V)}{V^{c_1/2}}$. 
We have: 
\begin{eqnarray}
\hspace{-.2in}&&\prob{\|\hat{\bv{\pi}}^m(t)-\hat{\bv{\pi}}^d(t)\|>\epsilon_d \,|\,\script{E}_1,  \mathcal{E}_2^c} \label{eq:exp-c-step-1} \\ 
\hspace{-.2in}&&\quad\,\,= \prob{\|\hat{\bv{\pi}}^m(t)-\hat{\bv{\pi}}^d(t)\|>\epsilon_d \,|\,\script{E}_1, \script{E}_{2a}^c}  \prob{ \script{E}_{2a}^c |  \script{E}_1, \script{E}_2^c}\nonumber\\
\hspace{-.2in}&&\qquad\quad+ \prob{\|\hat{\bv{\pi}}^m-\hat{\bv{\pi}}^d(t)\|>\epsilon_d \,|\,\script{E}_1, \script{E}_{2b}^c}  \prob{ \script{E}_{2b}^c |  \script{E}_1, \script{E}_{2}^c}. \nonumber
\end{eqnarray}
Let us relate $\prob{\|\hat{\bv{\pi}}^m(t)-\hat{\bv{\pi}}^d(t)\|>\epsilon_d \,|\,\script{E}_1, \script{E}_{2a}^c}$ to $\prob{\|\hat{\bv{\pi}}^m(t)-\hat{\bv{\pi}}^d(t)\|>\epsilon_d \,|\,\script{E}_1, \script{E}_{2}}$. 
Consider a $\hat{\bv{\pi}}^d(t)$ is such that $\|\hat{\bv{\pi}}^m(t)-\hat{\bv{\pi}}^d(t)\|>\epsilon_d$ given $\script{E}_1, \script{E}_{2}$. We note that there exist $i$ and $j$ such that $\pi^d_i(t)\leq\pi_i$ and $\pi^d_j(t)\geq\pi_j$. Then, we can always change $\hat{\bv{\pi}}^d(t)$ to $\tilde{\bv{\pi}}^d(t)$ by having one more sample for $j$ and one less sample for $i$ (this can be ensured with high probability since $d=O(\log^3(V))$). Since $\sigma=O(V^{-c_1/2})$ and $\epsilon_d = O(1/\log^3(V))$, we will have $\|\hat{\bv{\pi}}^m(t)-\tilde{\bv{\pi}}^d(t)\|>\epsilon_d$ given $\script{E}_1, \script{E}_{2a}^c$. Therefore, 
\begin{eqnarray*}
\hspace{-.2in}&& \prob{\|\hat{\bv{\pi}}^m(t)-\hat{\bv{\pi}}^d(t)\|>\epsilon_d \,|\,\script{E}_1, \script{E}_{2a}^c}  \\
\hspace{-.2in}&&\qquad\geq P_0\triangleq c_0\prob{\|\hat{\bv{\pi}}^m(t)-\bv{\pi}^d(t)\|>\epsilon_d \,|\,\script{E}_1, \script{E}_{2}}. \nonumber
\end{eqnarray*}
Here $c_0 = \min_i\pi_i/\max_j\pi_j$. This shows that the probability of having a change declared under $\script{E}_1, \script{E}_{2a}^c$ is more than a constant factor of that under $\script{E}_1, \script{E}_{2}$. 
As a result, using (\ref{eq:exp-c-step-1}) and the fact that $\prob{ \script{E}_{2a}^c |  \script{E}_1, \script{E}_2^c}\geq 1-O(V^{-3})$, 
\begin{eqnarray*}
\hspace{-.2in}&& \prob{\|\hat{\bv{\pi}}^m(t)-\hat{\bv{\pi}}^d(t)\|>\epsilon_d \,|\,\script{E}_1, \script{E}_{2}^c} \geq P_1, 
\end{eqnarray*}
where $P_1=c_1P_0$ and $c_1\geq c_0(1-O(V^{-3}))$. 
Thus, 
\begin{eqnarray}
\hspace{-.2in}&& \expect{|\script{C}| \,|\,\script{E}_1, \script{E}_{2}^c} \leq d/P_1. \label{eq:expected-size-e1e2c-0}
\end{eqnarray}
This is obtained by considering only testing for changes at multiples of $d$ slots. 

On the other hand, it can be shown that 
\begin{eqnarray}
\expect{|\script{C}| \,|\,\script{E}_1, \script{E}_{2}} \geq \Theta(1/P_1). \label{eq:expected-size-e1e2c-1}
\end{eqnarray}
This is so since, conditioning on $\script{E}_1, \script{E}_{2}$, samples in $\script{W}_d(t)$ evolves according to a Markov chain, with each state being a sequence of $d$ samples. Moreover, the total mass of the set of states resulting in $\|\hat{\bv{\pi}}^m(t)-\hat{\bv{\pi}}^d(t)\|>\epsilon_d$ is $P_0/c_0$ and that after $V^{2+c_1}$ time, the first $\script{W}_d(t)$ is drawn with the steady state probability (due to $S(t)$ being i.i.d.). Thus, the Markov chain is in steady state from then, showing that the time it takes to hit a violating state is $\Theta(1/P_1)$. 
%
Combining (\ref{eq:expected-size-e1e2c-1}) with (\ref{eq:expected-size-e1e2c-0}), we conclude that: 
\begin{eqnarray}
\hspace{-.2in}&& \expect{|\script{C}| \,|\,\script{E}_1, \script{E}_{2}^c} \leq d\expect{|\script{C}| \,|\,\script{E}_1, \script{E}_{2}}\leq 2d\expect{|\script{C}|}. \label{eq:expected-size-e1e2c}
\end{eqnarray}
The last inequality follows since $\prob{\script{E}_1, \script{E}_{2}}\geq 1-2V^{-3}$. 

Now consider the event $\mathcal{E}_1\cap \mathcal{E}_2$. 
Using the fact that $T_l=\Theta(V^{c_1})$, $\prob{\mathcal{E}_2\cap\mathcal{E}_1}\geq 1-O(V^{-3})$, 
and using almost verbatim arguments as in the proofs of Lemmas $8$ and $9$ in \cite{huang-rlc-15}, it can be shown that:\footnote{The fact that the event holds with  probability almost $1$ enables an analysis similar to that without conditioning.}  
\begin{eqnarray}
\hspace{-.3in}&&\expect{\sum_{t\in\mathcal{C}'}  [\mu_j(t) - A_j(t)]  \left.|\right. \mathcal{E}_1, \mathcal{E}_2 }\label{eq:bound-ex-3}\\
\hspace{-.3in}&&  \leq \expect{q_{js} - q_{je}\left.|\right. \mathcal{E}_1, \mathcal{E}_2} +\delta_{\max}(1+b_1\expect{|\mathcal{C}|\left.|\right.  \mathcal{E}_1, \mathcal{E}_2  }/V^{\log{V}}),  \nonumber
\end{eqnarray}
where $b_1=\Theta(1)$, and $q_{js}$ and $q_{je}$  denote the beginning and ending sizes of queue  $j$ during $\mathcal{C}'$, respectively. 
 
We first bound the $\expect{q_{js}}$. Conditioning on $\script{E}_1$, we see that there will be $T_l+d-w$ time until $W_m(t)=T_l$. 
%
Thus, $\expect{q_{js}}\leq \delta_{\max}b_2(V^{c_1}+d-w)$ for some $b_2=\Theta(1)$. 
Combining (\ref{eq:bound-ex-1}), (\ref{eq:bound-ex-2}), and (\ref{eq:bound-ex-3}), we obtain: 
\begin{eqnarray}
\hspace{-.2in}&&\expect{\text{cost} } \nonumber\\
\hspace{-.2in}&&\quad\leq V\log(V)\delta_{\max}\cdot (\frac{V^{2+c_1}}{V^{c^2\log(V)}} + \frac{3\expect{|\mathcal{C}|}}{V^{\log(V)/2}})\label{eq:sum-exp} \\ 
\hspace{-.2in}&&\qquad + V\log(V)\delta_{\max}\expect{|\mathcal{C}|\,|\,\mathcal{E}_1, \mathcal{E}_2^c}\cdot Me^{-4\log(V)^2} \nonumber \\
\hspace{-.2in}&&\qquad + (\delta_{\max}b_2(V^{c_1}+d-w) +w+1)\delta_{\max}V\log(V)     \nonumber \\
\hspace{-.2in}&&\qquad + V\log(V)\delta_{\max}(1+b_1\expect{|\mathcal{C}|\left.|\right.  \mathcal{E}_1, \mathcal{E}_2  }/V^{\log{V}}).\nonumber
\end{eqnarray} 
The term $(w+1)\delta_{\max}V\log(V)$ in the last $w+1$ slots after a change detection. 
%
Combining (\ref{eq:sum-exp}) with  (\ref{eq:renewal-cycle}),  (\ref{eq:length-exp}), and (\ref{eq:expected-size-e1e2c}),  we obtain (\ref{eq:drift-1overV}). 



(\textbf{Part (b) - Delay})  
From the above, we see that the event $\mathcal{E}_1\cap\mathcal{E}_2$ happens with probability at least $1-O(1/V^3)$. 
Hence, we only need to show that most packets that arrive during the $\mathcal{C}'$ intervals experience small delay, conditioning on $\mathcal{E}_1\cap\mathcal{E}_2$. 

Denote $t_s$ and $t_e$ the beginning and ending slots of $\mathcal{C}'$. 
Using (\ref{eq:event3-prob}) and Lemma \ref{lemma:estimate-error}, we get that with probability at least $1- 2V^{-3}$, 
\begin{eqnarray}
\| \bv{\gamma}^*(t) - \bv{\gamma}^*\|\leq d_{\gamma}\triangleq4b_0V^{1-c_1/2}\log(V).  \label{eq:beta-rate-general}
\end{eqnarray} 
Define 
\begin{eqnarray}
\hat{\bv{\theta}} \triangleq \bv{\gamma}^* - (\bv{\gamma}^*(t) - \bv{\theta})_+, \label{eq:theta-hat-def}
\end{eqnarray} 
%
we see from Lemma \ref{lemma:prob_multi_con} that whenever $\|  \bv{q}(t) - \hat{\bv{\theta}} \| > G$,  which is equivalent to $\|\bv{Q}(t) -  \bv{\gamma}^*\|>G$, 
\begin{eqnarray*}
\expect{\|  \bv{q}(t+1) - \hat{\bv{\theta}} \| \left.|\right. \bv{q}(t)} \leq \|  \bv{q}(t) - \hat{\bv{\theta}} \| - \eta,
\end{eqnarray*}
for the same $G=\Theta(1)$ and $\eta=\Theta(1)<\eta_1$ in Lemma \ref{lemma:prob_multi_con}.\footnote{This is due to conditioning on $\mathcal{E}_1\cap\mathcal{E}_2$.} 
Using (\ref{eq:beta-rate-general}) and $\theta$ in (\ref{eq:theta-value}), we see that $\hat{\bv{\theta}} = \Theta(d_{\gamma}\log(V) + \log(V)^2)$. 
Therefore, using Theorem 4 in \cite{huang-learning-sig-14}, if we assume that $\mathcal{C}'$ never ends,    
\begin{eqnarray}
\expect{T_{G}(\bv{q}(t))}\leq b_3d_{q}/\eta, \label{eq:conv-poly-tg1}
\end{eqnarray}
where $b_3=\Theta(1)$, $d_{q} = \|\hat{\bv{\theta}} - \bv{q}(t_s)\|$ and $T_{G}(\bv{q}(t))\triangleq \inf\{t-t_s: \| \bv{q}(t) - \hat{\bv{\theta}} \|\leq G\}$. 
Note that this is after $W_m(t)=T_l$ in \plc{}, which happens after $T_c=d-w+T_l$ slots from the beginning of the interval. 
By Markov inequality, 
\begin{eqnarray}
 \prob{T_{G}(\bv{q}(t)) +T_c  > (b_3d_q/\eta+d-w+T_l)V} =\frac{1}{V}.\label{eq:poly-prob-dropping}
\end{eqnarray}
%
Denote $\mathcal{E}_3\triangleq \{T_{G}(\bv{q}(t))+T_c(t)\leq (b_3d_q+d-w+T_l)V\}$ and let $t^*$ the first time after $t_s$ that $Y(t)\triangleq\|  \bv{q}(t) - \hat{\bv{\theta}} \|\leq G$. Following an argument almost identical to the proof of Theorem 1 in \cite{huangneely_dr_tac}, we obtain that: 
\begin{eqnarray}
\sum_{t=t^*}^{t_{e}}\frac{\nu\eta}{2}\expect{e^{\nu Y(t)}}\leq (t_{e}-t^*)e^{2\nu\sqrt{r}\delta_{\max}} + e^{\nu Y(t^*)}, \label{eq:poly-prob-dropping-detail0}
\end{eqnarray}
where $\nu \triangleq \frac{\eta}{\delta_{\max}^2 +\delta_{\max} \eta/3 }=\Theta(1)$. 
Define $b_4\triangleq 2e^{2\nu \sqrt{r}\delta_{\max}}/\nu \eta=\Theta(1)$ and $b_5\triangleq e^{\nu Y(t^*)}\leq e^{\nu G}=\Theta(1)$, and choose $m=\log(V)^2$. We have from  (\ref{eq:poly-prob-dropping-detail0}) that: 
\begin{eqnarray} 
\hspace{-.3in}&&\frac{1}{t_{e}-t_s}\sum_{t=t_s}^{t_{e}}\prob{Y(t)>G+m}  \label{eq:prob-queue-bdd}\\
\hspace{-.3in}&&\qquad  \leq \frac{1}{t_{e}-t_s}(\sum_{t=t*}^{t_{e}}\prob{Y(t)>G+m} +(t^*-1-t_s))\nonumber\\
\hspace{-.3in}&&\qquad \leq  \big[(b_4(t_e-t^*))V^{-\log(V)} + b_5+ (t^*-t_s)\big]/(t_{e}-t_s) \nonumber\\
\hspace{-.3in}&&\qquad = O(\frac{(b_3d_q+d-w+T_l)V +1}{V^{2+c_1}}) = O(1/V). \nonumber 
\end{eqnarray}
Thus, the above implies that, given the joint event  $\mathcal{E}_1\cap\mathcal{E}_2$, which happens with probability $1-O(1/V^3)$, 
%
the fraction of packets enter and depart from each $q_j(t)$ when $\|  \bv{q}(t) - \hat{\bv{\theta}} \|\leq G$ is given by 
 $(1-O(1/V))(1-O(1/V))$, i.e.,  $1- O(\frac{1}{V})$. 
 %
This means that they enter and depart when $q_j(t)\in[\hat{\theta} - G - \log(V)^2, \hat{\theta} + G + \log(V)^2]$ (due to LIFO), which implies that their average delay in the queue is $O(\log(V)^2)$. 

(\textbf{Part (c) - Dropping}) First, conditioning on $\mathcal{E}^c_{1a}$, which happens probability $V^{-\log(V)/2}$, we see that the algorithm drops at most $O(V^{2+c_1})$ packets in this case. 
%

Now consider when $\mathcal{E}_1$ takes place, and denote as above by $t_s$ and $t_e$ the starting and ending timeslots of a cycle. 
In this case, from the rules of \ade, we see that rule (ii) is inactive, since if it is satisfied at time $T_l$, it remains so because $\hat{\bv{\pi}}^m(t)$ remains unchanged until the cycle ends. Hence, the only case when an interval ends is due to violation rule (i). Let us suppose the interval ends because at some time $t'$, we have $\|\hat{\bv{\pi}}^m(t') - \hat{\bv{\pi}}^d(t') \|>\epsilon_d$. We know then \plc{} drops all packets at time $t'+w+1$, i.e., $\bv{q}(t'+w+1)$. 

We now bound $\expect{\bv{q}(t'+w+1)}$. To do so, consider the time $t^*=t'-2d$. We see then $\bv{q}(t^*)$ and all queue sizes before $t^*$ are independent of $\hat{\bv{\pi}}^d(t')$. Also, $\sum_jq_j(t'+w+1)\leq \sum_jq_j(t^*) + r(2d+w+1)\delta_{\max}$. 

Consider the time interval from when $W_m(t)=T_l$ utill $t^*$ and consider two cases, (i) $e_w=\Omega(V^{-c/2})$ and (ii) $e_w=O(V^{-c/2})$. 
In the first case, we see that $T_l=V^c$. Thus,  $q_j(t_s+T_l)\leq\delta_{\max}V^c$.

In the second case, since $e_w=O(V^{-c/2})$, $T_l=e_w^{-2}$. We have from Lemma \ref{lemma:estimate-error} that before time $T_l$, the estimated multiplier $\|\bv{\gamma}^*(t) -  \bv{\gamma}^* \|\leq Ve_w=O(V^{1-c/2})$. As a result, using the definition of $\hat{\bv{\theta}}$ in (\ref{eq:theta-hat-def}) and denoting $Z(t) = \|(\bv{q}(t) - \hat{\bv{\theta}})_+\|$, we see that whenever $Z(t)\geq G$, $\expect{|Z(t+1) - Z(t)|\,|\,Z(t)}\leq -\eta$. It can also be checked that the other conditions in  Lemma \ref{lemma:queue-bdd-drift} are satisfied by $Z(t)$, and $\bv{q}(t_s)=\bv{0}$ and $Z(0)=0$. 
Thus, 
\begin{eqnarray}
\expect{Z(T_l)} \leq G+\sqrt{r}\delta_{\max} +O(1). \label{eq:z-bdd}
\end{eqnarray}
Thus, $\expect{\bv{q}(t_s+T_l)}=O(V^{1-c_1/2})$. 
Combining the two cases, we have $\expect{\bv{q}(t_s+T_l)}=O(V^{1-c_1/2}   +V^{c}) = O(V)$. 

After $t_s+T_l$, the distribution $\hat{\bv{\pi}}^m(t)$ is used to compute the multiplier. Since $T_l =\max(V^c, e_w^{-2})$, we see that the argument above similarly holds. 
%
Thus, using Lemma \ref{lemma:queue-bdd-drift}, we see that $\expect{\bv{q}(t^*)} = O(V)$, which implies $\expect{\bv{q}(t'+w+1)} = O(V+d)$. 
Therefore, packets will be dropped no more than every $V^{2+c_1}$ slots, and at every time we drop no more than $O(V)$ packets on average. 

Finally, consider given $\mathcal{E}^c_{1b}$. 
Using (\ref{eq:length-exp}) and (\ref{eq:cond-cycle-foo}), we note that conditioning on $\mathcal{E}^c_{1b}$, the cycle lasts no more than $3 \expect{|\mathcal{C}|}$ on average, which means that the number of packets dropped is at most $O(\expect{|\mathcal{C}|})$ every cycle on average.  Moreover,  using (\ref{eq:prob-e-1b-c}), we see that this happens with probability $O(V^{-3})$. 

The result follows by combining the above cases. 
\end{proof}
 
\section*{Appendix C - Proof of Theorem \ref{theorem:plc-nonstationary}}\label{section:appendix-C}
In the following proof, we will refer to some steps in Appendix B, to avoid restating the steps.

\textbf{(Proof of Theorem \ref{theorem:plc-nonstationary})}  
%
We first have the following lemma to  show that if each $d_k\geq4d$, then \ade{}  keeps only $o(d)$ samples (timeslots) from the previous distribution in $\mathcal{W}_m(t)$ after change detection. 
This step is important, as if $\mathcal{W}_{m}(t)$ contains too many samples from a previous distribution interval, the distribution estimation $\hat{\bv{\pi}}^m(t)$ can be inaccurate and lead to a high false-negative rate, which in turn affects performance during $\mathcal{I}_k$. 
%
The proof of the lemma is given at the end of this section. 
\begin{lemma}\label{lemma:samples}
Under the conditions of Theorem \ref{theorem:plc-nonstationary}, with probability $1- O(V^{-3\log(V)/4})$ only $o(d)$ samples from $\mathcal{I}_{k-1}$ remain in $\mathcal{W}_{m}(t)\cup\mathcal{W}_{d}(t)$ for $t\geq t_k+d$. $\Diamond$
\end{lemma}

We now prove Theorem \ref{theorem:plc-nonstationary}. 

\begin{proof} (Theorem \ref{theorem:plc-nonstationary})  We first have from Lemma \ref{lemma:detection} that with probability at least $1-V^{-3\log(V)/4}$ ($\delta=V^{-3\log(V)/4}$), distribution change will be detected before $t_k+d-w$. Denote this event by $\mathcal{E}_4$. 
\begin{figure}[ht]
\begin{center}
\vspace{-.1in}
\includegraphics[width=3.2in, height=0.8in]{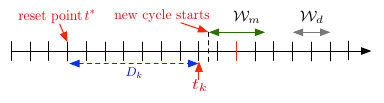}
\vspace{-.25in}
\caption{\small{Intervals in a non-stationary system.}}
\label{fig:timeline-structure2}
\end{center}
\vspace{-.15in}
\end{figure}


(\textbf{Part (a) - Utility}) 
Using Lemma \ref{lemma:samples}, we see that $o(d)$ samples will remain in $\mathcal{W}_m(t)$. This implies that when $V$ is large and $W_{m}(t)=d$, with probability $1-O(V^{-\log(V)/2})$, 
\begin{eqnarray}
|\hat{\pi}^m_{i}(t) - \pi_{ki}   | \leq \epsilon_d/8,  \,\forall\, i, \label{eq:od-consequence}
\end{eqnarray}
where $\hat{\bv{\pi}}^m(t)$ is the distribution in window $\mathcal{W}_{m}(t)$ (can contain  timeslots from the previous interval). This shows that the empirical distribution of $\mathcal{W}_{m}(t)$ is close to the true distribution even though it may contain samples from $\script{I}_{k-1}$. 
%
Thus, as $W_{m}(t)$ increases, $\hat{\bv{\pi}}^{m}(t)$ will only become closer to $\bv{\pi}_{k}$, so that (\ref{eq:od-consequence}) holds whenever $\mathcal{W}_{d}(t)\subset \mathcal{I}_k$. 
Denote (\ref{eq:od-consequence}) event $\mathcal{E}_5$. 

Now use an argument similar to the proof of Lemma  \ref{lemma:detection}, we can show that:  
\begin{eqnarray}
\hspace{-.2in}&&\prob{\|\hat{\bv{\pi}}^d(t) - \hat{\bv{\pi}}^m(t)\|_{TV} \geq \epsilon_d} \leq V^{-\log(V)/3}. \nonumber
\end{eqnarray}
Hence, for each cycle $\mathcal{C}\subset\mathcal{I}_k$, if we denote $\mathcal{E}_6$ the event that \ade{} does not declare any distribution change in steps (i) and (ii) for $V^{1+a}T_l\log(V)$ slots, and $\mathcal{E}_2$ before equation (\ref{eq:event2-prob}) holds (note that (\ref{eq:event2-prob}) only requires that $W_m(t)=T_l$), 
we see that 
\begin{eqnarray} 
\prob{\mathcal{E}_6 } \geq 1- V^{-2}. 
\end{eqnarray}
%
This implies that $\mathcal{I}_k$ mostly likely only contains one cycle $\mathcal{C}$. 

Therefore, conditioning on $\mathcal{E}_4\cap\mathcal{E}_5\cap\mathcal{E}_6$, which happens with probability $1-O(V^{-2})$ and implies that for cycle $\mathcal{C}'$, $\bv{q}(t_s)=\Theta(D_k+T_l+d-w)$, we have: 
\begin{eqnarray}
\hspace{-.1in}&& \expect{\text{cost} } \leq  r(D_k+T_l+d-w)b_2\delta_{\max}V\log(V)     \nonumber \\
\hspace{-.1in}&&\qquad\qquad + V\log(V)\delta_{\max}(1+b_1\expect{|\mathcal{C}| }/V^{\log{V}}).\nonumber 
\end{eqnarray} 
Applying the argument in the proof of  Theorem \ref{theorem:plc-stationary}, we see that $\frac{1}{d_k}\sum_{t=0}^{d_k-1}\expect{\Delta_1(t)} = O(\frac{D_k\log(V)}{T_lV^a})$. Hence, the result follows. 
%

(\textbf{Part (b) - Queue}) From the above, we see that at time $t_k$, $\bv{q}(t_k)=O(D_k)$. We also know that the current cycle $\mathcal{C}$ will start no later than $t_k+d-w$ with probability $1-O(V^{-3\log(V)/4})$, in which case $\bv{q}(t_s)=O(D_k+d-w)$. 

%

Since the system is polyhedral with $\rho$, using an argument similar to the proof for Part (c) of Theorem \ref{theorem:plc-stationary} (which is possible since the argument there applies to finite intervals),   if we define $\tilde{\bv{\theta}} = \Theta((\min(V^{1-c/2}, Ve_w)+1)\log^2(V) + D_k+d-w)$ and $Z(t) =\|(\bv{q}(t) - \tilde{\bv{\theta}})_+\|$, then throughout $t\in[t_s, t_{k+1}-1]$, 
\begin{eqnarray} 
\expect{Z(t)} \leq G+\sqrt{r}\delta_{\max} +O(1). 
\end{eqnarray}
Therefore, Part (b)  follows.  
\end{proof}

Here we provide the proof for Lemma  \ref{lemma:samples}. 

\begin{proof} (Lemma \ref{lemma:samples}) 
Consider time $t=t_k-w$. We have the following cases. 

(i) $W_{d}(t) = d$ and $W_{m}(t) < d$. Since $d_{k-1}\geq4d$, we see that the change point $t_k$ will be detected with probability at least $1-\delta$ at time $t'\leq t+d$, because $\mathcal{W}_{d}(t')$ will contain samples from $\bv{\pi}_k$ while $\mathcal{W}_{m}(t')$ will contain samples from $\bv{\pi}_{k-1}$ (Note that although this is conditioning on $W_{d}(t) = d$ and $W_{m}(t) < d$, since at this point no statistical comparison will be assumed, it is independent of the realizations in the two windows). Moreover, all samples from $\mathcal{I}_{k-1}$ will be removed and not remain in $W_{m}(t)$ and $W_{d}(t)$, while at most $w+1$ samples from $\mathcal{I}_k$ will be discarded.


(ii) $W_{d}(t) = d$ and $W_{m}(t)\geq d$. In this case, if a change is declared, then we turn to case (iii). 
Otherwise, since the samples in $\mathcal{W}_{m}(t)$ are drawn from $\bv{\pi}_{k-1}$,  
 we have:  
 \begin{eqnarray}
\prob{\|\hat{\bv{\pi}}^m(t) - \bv{\pi}_{k-1}\| \leq \epsilon_d/2} \geq 1- V^{-3\log(V)/4}. \label{eq:k00-bdd}
\end{eqnarray}
Now suppose no change is detected till time $t+d$. Then $W_m(t+d)\geq d$. 
Denote $\mathcal{E}_7\triangleq\{ \|\hat{\bv{\pi}}^m(t) - \bv{\pi}_{k-1}\| \leq \epsilon_d/2  \}$. 
Conditioning on $\mathcal{E}_7$ and using (\ref{eq:k00-bdd}), we have:  
 \begin{eqnarray}
\prob{\|\hat{\bv{\pi}}^m(t+d) - \bv{\pi}_{k-1}\| \leq \frac{\epsilon_d}{2} \,|\,\mathcal{E}_7} \geq 1- 2V^{-3\log(V)/4}. \label{eq:k01-bdd}
\end{eqnarray}
The inequality follows since $\prob{  \mathcal{E}_7 } \geq 1- V^{-3\log(V)/4}$. 
Now  $\mathcal{W}_{d}(t+d)$ contains only samples from $\bv{\pi}_k$, in which case we similarly have: 
\begin{eqnarray}
\prob{\|\hat{\bv{\pi}}^d(t+d) - \bv{\pi}_{k}\| \leq \epsilon_d/2 } \leq 1- V^{-3\log(V)/4}. \label{eq:k-bdd}
\end{eqnarray} 
 Since the state realizations in (\ref{eq:k01-bdd}) and (\ref{eq:k-bdd}) are independent, we conclude that with probability $1- 3V^{-3\log(V)/4}$, a change will be declared before $t_k$ and all samples from $\mathcal{I}_{k-1}$ will be removed and not remain in  $\mathcal{W}_{m}(t)\cup\mathcal{W}_{d}(t)$. 


(iii) $W_{d}(t) < d$. We argue that with high probability, at most $o(d)$ samples can remain at time  $t_k+2d - W_{d}(t)$. 
First, note that $W_{d}(t) < d$ only occurs when a detection has been declared at a time $t+w-d\leq t'\leq t$. Thus, if $t+w-t'=o(d)$, then we are done. Otherwise suppose $t+w-t'=\alpha d$ for $\alpha=\Theta(1)$. If they are removed, then at time $t'+2d$,  $\mathcal{W}_{m}(t'+2d)$ contains samples with mixed distribution $\bv{\pi}'=\alpha\bv{\pi}_{k-1}+(1-\alpha)\bv{\pi}_k$ and $\mathcal{W}_{d}(t'+2d)$ containing samples with  distribution $\bv{\pi}_k\neq \bv{\pi}'$. 
Similar to case (i), the condition  $W_{d}(t) < d$ is independent of the state realizations in the two windows. 
Using Lemma \ref{lemma:detection} (it can be checked that the conditions in the lemma are satisfied), we see that this will be detected by \ade{} with probability $1-\delta$ with a large $V$. 

Combining all three cases completes the proof. 
\end{proof}

\section*{Appendix D - Proof of Theorem \ref{theorem:perfect-plc}}
\textbf{(Proof of Theorem \ref{theorem:perfect-plc})}   
We prove Theorem \ref{theorem:perfect-plc} here. 
\begin{proof}
(\textbf{Part (a) - Stationary}) The results follow from the fact that when $e_w=0$, \plc{} is equivalent to \olac{} in \cite{huang-learning-sig-14} with perfect statistics. Hence the results follow from Theorems $1$ and $2$ in \cite{huang-learning-sig-14}.

(\textbf{Part (b) - Non-Stationary}) We first see that at time $t$, \ade{} detects distribution change in time $t+w$ through step (ii) with probability $1$. 
Then, after time $t_k+d-w$, $\bv{\pi}_a(t)=\bv{\pi}_k$ and we see that whenever $Z(t)\triangleq\|\bv{q}(t)  - \bv{\theta}\|>G$ for $\theta=2\log^2(V)$ and $G=\Theta(1)$,  
\begin{eqnarray} 
\expect{Z(t+1)\,|\, \bv{q}(t) } \leq Z(t) -\eta. 
\end{eqnarray}
Denote $b_6 = \frac{1}{r_z}\log(e^{r_zr\delta_{\max}-\rho_z}/(1-\rho_z))$. 
We want to show via induction that for all $k$, 
\begin{eqnarray} 
\expect{\sum_jq_j(t_k)} \leq q_{th}\triangleq 2r\log^2(V)+b_6+ 2G+dr\delta_{\max}. \label{eq:queue-bdd-perfect-predict}
\end{eqnarray} 
%
First, it holds for time zero. Suppose it holds for interval $\mathcal{I}_k$. We now show that it also holds for interval $k+1$. 

To do so, first we see that during time $[t_k, t_k+d-w]$, there can be an increment of $q_j(t)$ since $\bv{\pi}_a(t)$ during this interval is a mixed version of $\bv{\pi}_{k-1}$ and $\bv{\pi}_{k}$. Thus, 
\begin{eqnarray}
\expect{\sum_jq_j(t_k+d)} \leq q'_{th}\triangleq q_{th} + dr\delta_{\max}. 
\end{eqnarray} 
Using Lemma \ref{lemma:queue-bdd-drift}, we have: 
\begin{eqnarray*} 
\hspace{-.2in}&&\expect{e^{r_z  Z(t_{k+1}-d) }} \\
\hspace{-.2in}&&\qquad\leq \frac{e^{r_zr\delta_{\max}-\rho_z}}{1-\rho_z}e^{r_zG} + (e^{r_zq'_{th}} -b_6e^{r_zG})\rho_z^{d_k-2d}. 
\end{eqnarray*}
Using the definition of $q_{th}$ and the fact that $d_k\geq d\log^2(V)$, we have that for a large $V$, $(e^{r_zq'_{th}} -b_6e^{r_zG})\rho^{d_k-2d} \leq G$. Thus, 
\begin{eqnarray} 
\expect{Z(t_{k+1}-d) } \leq b_6 +2G, 
\end{eqnarray}
 which implies $\expect{\sum_jq_j(t_{k+1} - d)} \leq 2r\log^2(V) + b_6 +2G$. It thus follows that $\expect{\sum_jq_j(t_{k+1} - d)} \leq q_{th}\leq b_7\log^4(V)$ for some $b_7=\Theta(1)$. 
 
Having established this result, using an argument similar to that in the proof of Theorem \ref{theorem:plc-nonstationary}, we have: 
\begin{eqnarray}
\hspace{-.1in}&& \expect{\text{cost} } \leq  b_7\log^4(V)\cdot V\log(V)     \nonumber \\
\hspace{-.1in}&&\qquad\qquad + V\log(V)\delta_{\max}(1+b_1\expect{|\mathcal{C}| }/V^{\log{V}}).\nonumber 
\end{eqnarray} 
Using $d_k\geq V^{1+\epsilon}$, we see that Part (b) follows. 
%
%
\end{proof}

\section*{Appendix E - Proof of Theorem \ref{theorem:plc-conv-stationary}} 

\textbf{(Proof of Theorem \ref{theorem:plc-conv-stationary})}    
Here we prove the convergence results. We sometimes drop the subscript $k$ when it is clear. 
\begin{proof} (Theorem \ref{theorem:plc-conv-stationary}) 
First, when $e_w=0$, we see that for any interval $\mathcal{I}_k$, for all time $t\geq t_k+d$, $\bv{\pi}_a(t)=\bv{\pi}_k$, and $\bv{\gamma}^*(t) = \bv{\gamma} - \bv{\theta}$. Using Lemma 5 in \cite{huang-learning-sig-14} and the fact that $d=O(\log^2(V))$, we have: 
\begin{eqnarray*}
&& \expect{T_G} = \expect{ \expect{T_G \,|\, \bv{q}(t_k)} } \\
&&\qquad\quad\, \stackrel{(*)}{=} \expect{\Theta( \| \bv{q}(t_k)  - \bv{\theta} \| )   }
\stackrel{(**)}{=} \Theta(\log^4(V)). 
\end{eqnarray*}
Here (*) follows from Lemma 5 in \cite{huang-learning-sig-14} and (**) follows from (\ref{eq:queue-bdd-perfect-predict}). 

Consider the other case  $e_w>0$. 
Using Lemma \ref{lemma:samples}, we see that with probability at least $1-V^{-3}$, \plc{} detects distribution change before time $t_k+d$. 
Recall the event $\mathcal{E}_1$  that \ade{} does not declare change in the first $V^{2+c_1}$ slots from the proof of Theorem \ref{theorem:plc-stationary}, where $c_1$ is such that $T_l=V^{c_1}$. Note that this implies $\{ d_{m} \geq V^{2+c_1}\}$). 
From (\ref{eq:length-prob}), we know that:  
\begin{eqnarray}
\prob{\mathcal{E}_1} \geq 1-V^{-3}. 
\end{eqnarray}
Conditioning on $\mathcal{E}_1$, the time it takes to achieve $||\bv{Q}(t)-\bv{\gamma}^*||\leq G$ is no more than the sum of (i) the time it takes to reach $W_m(t)=T_l$, and (ii) the time it takes to go from the estimated multiplier $\bv{\gamma}^*(t) - \bv{\theta}$ to $\bv{\gamma}^*$. 
%
Denote $\mathcal{E}_8(t) = \{\|\bv{\pi}^m(t) -  \bv{\pi} \|_{TV} \leq 2M\log(T_l)T_l^{-1/2}\}$. 
When $W_m(t)=T_l$, we have 
\begin{eqnarray} 
\prob{\mathcal{E}_8(t)} \geq 1- O(MT_l^{-2\log(T_l)}), \label{eq:random-eq-00}
\end{eqnarray}
in which case $\|\bv{\gamma}^*(t) -\bv{\gamma}^* \| = \Theta(\frac{V\log(V)}{\sqrt{T_l}})$. 
As in the proof of Theorem \ref{theorem:plc-nonstationary}, we see that when $W_m(t)=T_l$, $\bv{q}(t)=O(D_k+T_l+d)$, which implies that $\|\bv{Q}(t) - \bv{\gamma}^*\| = \Theta((1+\frac{V}{\sqrt{T_l}})\log^2(V)+T_l+ D_k+d)$. Using Lemma 5 in \cite{huang-learning-sig-14} again, we see that if \ade{} always does not declare change,  
\begin{eqnarray}
\expect{T_G} = O(\theta+T_l+ D_k+d). 
\end{eqnarray} 
Using Markov inequality, we see that: 
\begin{eqnarray}
\prob{T_G \geq V^{2+c_1} }\leq O(V^{-1-c_1} +D_kV^{-2-c_1} ). 
\end{eqnarray}
Thus, with probability $1-O(V^{-1-c_1} +D_kV^{-2-c_1})$, convergence occurs before $V^{2+c_1}$. This proves (\ref{eq:convergence-time1}). 

%

To prove (\ref{eq:convergence-time2}), define $G_1=\Theta(D_k+2\log(V)^2(1+Ve_w))$. Then, we see from Lemma \ref{lemma:samples} that with probability $1- O(V^{-3\log(V)/4})$, distribution change will be detected before $t'\leq t_k+d$. At that time, we have 
$\|\bv{\gamma}^*(t)-\bv{\gamma}^*\| = O(Ve_w)$. 
Combining this with the fact that $\bv{q}(t')=O(D_k+d)$, we see that (\ref{eq:convergence-time2}) follows. 
This completes the proof. 
\end{proof}

\end{document}